\documentclass[11pt]{article}
\usepackage{amssymb, amsthm, amsmath, amscd}
\newtheorem{thm}{Theorem}[section]
\newtheorem{lem}[thm]{Lemma}

\newtheorem{cor}[thm]{Corollary}
\newtheorem{NN}[thm]{}
\theoremstyle{definition}\newtheorem{df}[thm]{Definition}
\theoremstyle{definition}
\theoremstyle{definition}

\renewcommand{\phi}{\varphi}

\newcommand{\Z}{\mathbb{Z}}
\newcommand{\Q}{\mathbb{Q}}
\newcommand{\R}{\mathbb{R}}
\newcommand{\C}{\mathbb{C}}
\newcommand{\T}{\mathbb{T}}

\newcommand{\Aff}{\operatorname{Aff}}

\newcommand{\hm}{homomorphism}
\newcommand{\dt}{\delta}
\newcommand{\ep}{\epsilon}

\newcommand{\andeqn}{\,\,\,{\rm and}\,\,\,}
\newcommand{\rforal}{\,\,\,{\rm for\,\,\,all}\,\,\,}
\newcommand{\CA}{$C^*$-algebra}
\newcommand{\SCA}{$C^*$-subalgebra}

\newcommand{\beq}{\begin{eqnarray}}
\newcommand{\eneq}{\end{eqnarray}}
\newcommand{\tforal}{\,\,\,\text{for\,\,\,all}\,\,\,}
\newcommand{\tand}{\,\,\,\text{and}\,\,\,}
\newcommand{\aff}{{\rm Aff}}
\newcommand{\p}{{\mathfrak{p}}}
\newcommand{\q}{{\mathfrak{q}}}
\newcommand{\fr}{{\mathfrak{r}}}

\title{Exponential rank and exponential length for ${\cal Z}$-stable simple \CA s }
\author{Huaxin Lin
 }
\date{}

\begin{document}

\maketitle

\begin{abstract}
Let $A$ be a unital separable simple ${\cal Z}$-stable \CA\,
 which has rational tracial rank at most one and let $u\in U_0(A),$ the connected component of the unitary group
of $A.$ We show that, for any $\ep>0,$ there exists a self-adjoint element $h\in A$ such that
\beq\label{ab-1}
\|u-\exp(ih)\|<\ep.
\eneq
The lower bound of $\|h\|$  could be as large as one wants. If $u\in CU(A),$ the closure of the commutator subgroup of the unitary group, we prove that there exists a self-adjoint element $h\in A$ such that
\beq\label{ab-2}
\|u-\exp(ih)\| <\ep\andeqn \|h\|\le 2\pi.
\eneq
Examples are given that the  bound $2\pi$ for  $\|h\|$  is the optimal
in general.  For the Jiang-Su algebra ${\cal Z},$ we show that, if  $u\in U_0({\cal Z})$ and $\ep>0,$ there exists a real number
$-\pi<t\le \pi$ and a self-adjoint element $h\in {\cal Z}$ with $\|h\|\le 2\pi$ such that
$$
\|e^{it}u-\exp(ih)\|<\ep.
$$

\end{abstract}

\section{Introduction}

Let $A$ be a unital \CA\, and let $U_0(A)$ be the connected component of unitary group of $A$ containing the identity. Suppose that $u\in U_0(A).$ Then $u$ is a finite product of exponentials, i.e.,
$u=\prod_{k=1}^n \exp(ih_k),$ where $h_k$ is a self-adjoint element
in $A.$ One of the interesting questions about the unitary group
of a \CA\, is when $u$ is an exponential? Or more interesting question
is when $u$ is a norm limit of exponentials. If $u\in U_0(A)$ and
$\{u(t): t\in [0,1]\}\subset U_0(A)$ is a continuous path connecting $u$ to the identity,
one may ask how long the length of the path could be.
Exponential rank and exponential length had been extensively studied
(see \cite{Ri}, \cite{RP}, \cite{Ph-fu}, \cite{Ph-2}, \cite{Zh-1}, \cite{Zh-2},  \cite{GLexp}, \cite{Ph-3}, \cite{Ph-4}, \cite{Lnexp},
\cite{Ph-1}, \cite{Thoms-1},  etc. --an incomplete list).

Exponential length and rank have  played, inevitably, important roles in the study
of structure of \CA s, in particular,
in the Elliott program, the classification of amenable
\CA s by $K$-theoretic invariant. The renew interest and direct motivation of this study is the recent research project to study the stable Jiang-Su algebra and its multiplier algebra. It turns out that exponential length
again plays an essential role there.

Let us briefly summarize some facts about exponential rank and length
for unital (simple and  amenable) \CA s in the center of
the Elliott program. It was shown by N. C. Phillips (\cite{Ph-2})
that the exponential rank of a unital purely infinite simple \CA\, is $1+\ep$ and its exponential length is $\pi.$ In fact,
this holds for any unital \CA s of real rank zero (\cite{Lnexp}).
 In other words,
if $u\in U_0(A),$ where $A$ is a unital \CA\, of real rank zero,
then, for any $\ep>0,$ there exists a self-adjoint  element $h\in A$ with
$\|h\|\le \pi$ such that
$$
\|u-\exp(ih)\|<\ep.
$$
These are smallest numbers that one can get.  When $A$ is not of real rank zero, the situation is very different.
For example, if $A$ is a unital simple AH-algebra with slow dimension growth,
then ${\rm cer}(A)=1+\ep.$  But ${\rm cel}(A)=\infty,$ whenever $A$ does not have real rank zero (Theorem 3.5 of \cite{Ph-1}). Recently it was shown (\cite{Lnexp2}) that
${\rm cer}(A)\le 1+\ep$ for any unital simple \CA\, $A$ with tracial rank at most one (without assuming the amenability).

The classification of unital simple amenable \CA s now includes classes
of \CA s far beyond \CA s mentioned above. In fact unital separable simple amenable ${\cal Z}$-stable \CA s which are rationally tracial rank at most one and satisfy the UCT can be classified by the Elliott invariant (\cite{Lninv}).
In this paper, we show that, if $A$ is ${\cal Z}$-stable, i.e.,
$A\otimes {\cal Z}=A,$ has rational tracial rank at most one,
i.e., $A\otimes U$ has tracial rank at most one for some
infinite dimensional UHF algebra $U,$ and $u\in U_0(A),$ then, for any $\ep>0,$  there exists a self-adjoint element $h\in A$ such that
\beq\label{Int-1}
\|u-\exp(ih)\|<\ep.
\eneq
However, in general, there is no control of the norm of $h.$ In fact,
${\rm cel}(A)=\infty,$ i.e., the exponential length of $A$ is infinite.

In the study of classification of simple amenable \CA s, one relies on
a fact that exponential length for unitaries in $CU(A),$ i.e., the closure of the commutator subgroup of $U_0(A)$ is often bounded.
It seems to suggest that, for exponential length of a unital \CA, it is the exponential length of unitaries in $CU(A)$ that needs
to be computed.  So the question is what is the norm bound for the above $h$ when $u$ is in $CU(A).$
We show that, if $A$ is a unital separable simple \CA\, with tracial rank at most one, and $u\in CU(A),$ then (\ref{Int-1}) holds and
$h$ can be chosen so that $\|h\|\le 2\pi.$ Furthermore, we also  prove this
holds for any unital separable simple ${\cal Z}$-stable \CA\, $A$ such that $A\otimes U$ has tracial rank at most one.   Originally, we would like to have the length
at minimum so that $\|h\|$ could be controlled by $\pi.$ The reason
that we have the bound $2\pi$ instead of $\pi$ is not a technical difficulty in the proof.  The reason is that $2\pi$ is the optimal estimate, a fact that we did not realized  which prevents us to have this research done earlier. We show in this paper that, in general,
for a unital simple AH-algebra (or even AI-algebra) $A,$ for any $\sigma>0,$  there are unitaries $u\in U_0(A)$ such that $\|h\|\ge 2\pi-\sigma$ if (\ref{Int-1}) holds for some sufficiently small $\ep.$
What is more surprising at the first was the answer to the question
how long the exponential length of unitaries in $U_0({\cal Z})$ is,  where
${\cal Z}$ is the Jiang-Su algebra, the projectionless simple ASH-algebra with $K_0({\cal Z})=\Z$ and $K_1({\cal Z})=\{0\}.$ It seems that, among experts, one expects the exponential length of ${\cal Z}$  to  be
infinite since ${\cal Z}$ does not have real rank zero.  However, we find that ${\rm cel}({\cal Z})\le 3\pi.$ In fact, we  prove that for any unitary  $u\in U_0({\cal Z}),$  there exists $-\pi<t<\pi$ satisfying the following: for any $\ep>0,$ there exists a self-adjoint  element $h\in {\cal Z}$ with
$\|h\|\le 2\pi$  such that
$$
\|e^{it}u-\exp(ih)\|<\ep.
$$
  We actually
prove this for all unital separable simple  ${\cal Z}$-stable \CA s with a unique tracial state which are rationally tracial rank zero.
An application  of the estimate of exponential length
for those simple \CA s can be found in a subsequent joint work with
Ping Ng.

{\bf Acknowledgments}  The author would acknowledge that he is benefited with e-mail correspondences 
with 
Guihua Gong and N. Chris Phillips during the writing of this research.

\section{Notations}

\begin{df}\label{D1}
{\rm
Let $A$ be a unital \CA. We denote by $U(A)$ the unitary group
of $A.$ We denote by $U_0(A)$ the connected component of $U(A)$ containing the identity and $CU(A)$ the {\it closure} of the commutator subgroup of $U_0(A).$
If $u\in U(A),$ we use the notation ${\bar u}$ for its image in
$U(A)/CU(A).$

Let $u\in U_0(A).$ Denote by ${\rm cel}(u)$ the exponential length of $u$ in $A.$
In fact,
$$
{\rm cel}(u)=\inf\{\sum_{k=1}^n \|h_k\|: u=\prod_{k=1}^n \exp(i h_k):
h_k\in A_{s.a.}\}.
$$

Define
$$
{\rm cel}(A)=\sup\{{\rm cel}(u): u\in U_0(A)\}.
$$
Define
$$
{\rm cel}_{CU}(A)=\sup\{{\rm cel}(u): u\in CU(A)\}.
$$
If $u=\lim_{n\to\infty} u_n,$ where $u_n=\prod_{j=1}^k \exp(ih_{n,j})$ for some self-adjoint  elements $h_{n,j}\in A.$  Then
we write
$$
{\rm cer}(u)\le k+\ep.
$$
If $u=\prod_{j=1}^k \exp(h_j)$ for some $h_1,h_2,...,h_k\in A_{s.a.},$ we write
$$
{\rm cer}(u)\le k.
$$
If ${\rm cer}(u)\le k+\ep$ but ${\rm cer}(u)\not\le k,$ we write
${\rm cer}(u)=k+\ep.$ If ${\rm cer}(u)\le k$ but
${\rm cer}(u)\not\le (k-1)+\ep,$ we write ${\rm cer}(u)=k.$

By $T(A),$ we mean the tracial state space of $A$ and by
$\Aff(T(A))$ the space of all real affine continuous functions on $T(A).$ Let $\tau\in T(A).$ We also use $\tau$ for the trace $\tau\otimes Tr$
on $A\otimes M_n,$ where $Tr$ is the standard trace on $M_n.$

Denote by $\rho_A: K_0(A)\to \Aff(T(A))$ the positive \hm s defined by
$\rho_A([p])=\tau(p)$ for all projections $p\in M_n(A),$ $n=1,2,....$

}

\end{df}

\begin{df}\label{Ddet}

{\rm
Let $A$ be a unital \CA\, with $T(A)\not=\emptyset.$
Let $u\in U_0(A).$
Suppose that $\{u(t): t\in [0,1]\}$ is a continuous path of unitaries
which is also piece-wisely smooth such that $u(0)=u$ and $u(1)=1.$
Define de la Harp-Skandalis determinant as follows:
\beq\label{Ddet-1}
{\rm Det}(u):={\rm Det}(u(t)):=\int_{[0,1]}\tau({du(t)\over{dt}}w(t)^*) dt\tforal \tau\in T(A).
\eneq
Note that, if $u_1(t)$ is another continuous path which is piece-wisely smooth with $u_1(0)=u$ and $u_1(1)=1,$ Then
${\rm Det}((u(t))-{\rm Det}(u_1(t))\in \rho_A(K_0(A)).$
Suppose that $u, v\in U(A)$ and $uv^*\in U_0(A).$
Let $\{w(t): t\in [0,1]\}\subset U(A)$ be a piece-wisely smooth
and continuous path such that $w(0)=u$ and $w(1)=v.$
Define
$$
R_{u,v}(\tau)={\rm Det}(w(t))(\tau)=\int_{[0,1]}\tau({dw(t)\over{dt}}w(t)^*) dt\tforal \tau\in T(A).
$$
Note that $R_{u,v}$ is well-defined (independent of the choices of the path) up to elements in $\rho_A(K_0(A)).$

}

\end{df}

\begin{df}\label{Zpq}
{\rm Denote by $\Q$ the group of rational numbers. Let $\fr$ be a supernatural number. Denote by $M_\fr$ the UHF-algebra associated with $\fr.$
Denote by $\Q_\fr$ the group $K_0(M_\fr)$ with order as a subgroup
of $\Q.$

Denote by ${\cal Z}$ the Jiang-Su algebra (\cite{JS}) which is
a unital separable simple ASH-algebra with $K_0({\cal Z})=\Z$ and
$K_1({\cal Z})=\{0\}.$
Let $\p, \q$ be two relatively prime supernatural numbers of infinite type. Denote by
$$
{\cal Z}_{\p, \q}=\{f\in C([0,1], M_{\p\q}):
 f(0)\in M_\p\andeqn f(1)\in M_\q\}.
 $$
 Here we identify $M_\fr$ with $M_\fr\otimes 1$ as a subalgebra of
 $M_{\p\q}.$  One may write
 ${\cal Z}$ as a stationary inductive limit of ${\cal Z}_{\p,\q}$
 (see \cite{RW}).
 }
\end{df}

\begin{df}\label{trrank}
{\rm
Let $A$ be a unital simple \CA. We write $TR(A)=0$ if tracial rank of $A$ is zero. We write $TR(A)\le 1,$ if the tracial rank of $A$ is either zero or one (see \cite{Lnproclond}).

Denote by ${\cal A}_0$ the class of unital separable simple \CA s
$A$ such that $TR(A\otimes U)=0$ for some infinite dimensional UHF-algebra $U.$  Note that ${\cal Z}\in {\cal A}_0.$

Denote by ${\cal A}_1$ the class of unital simple separable \CA s $A$ such that $TR(A\otimes U)\le 1.$  We refer the reader to
(\cite{W}), (\cite{Lnappen}), (\cite{LN-1}), (\cite{LN-2}), (\cite{Lnhuasdo}), (\cite{Lninv}) and (\cite{LS})  for some further discussion of these \CA s.

}
\end{df}

\begin{df}\label{pointev}
{\rm
Let $A$ be a unital \CA\, and $C=C([0,1], A).$
Denote by $\pi_t: C\to A$ the point-evaluation: $\pi_f(f)=f(t)$ for all $f\in C.$

}
\end{df}

\begin{df}\label{Dmeas}

{\rm
Let $X$ be a compact metric space and let $\psi: C(X)\to \C$ be a state.  Denote by $\mu_{\psi}$ the probability Borel measure induced by $\psi.$
}

\end{df}

\section{Exponential rank}
The following could be easily proved directly. But it is a special case of 6.3 of \cite{Lninv}.

\begin{lem}\label{botext}
Let $\ep>0.$ There exists $\dt>0$ satisfying the following:
Suppose that $A$ is a unital separable simple \CA\, with $TR(A)\le 1$ and suppose
that $u\in U(A)$ with ${\rm sp}(u)=\T.$ Then, for any
$x\in K_0(A)$  with $\|\rho_A(x)\|<\dt$ and any $y\in K_1(A),$  there exists a unitary
$v\in A$  with $[v]=y$ in $K_1(A)$ such that
\beq\label{botext-1}
\|[u,\, v]\|<\ep\andeqn {\rm bott}_1(u,v)=x.
\eneq

\end{lem}

The following is also known and we state here for the convenience.
\begin{lem}\label{triv}
Let $A$ be a unital \CA\, with $T(A)\not=\emptyset.$ Let $u$ and $v$ be two unitaries in $A$ with $[u]=[v]$ in $K_1(A).$
Suppose that there is a unitary $w\in A$ such that
\beq\label{triv-1}
\|uw^*v^*w-1\|<2.
\eneq
Then,
\beq\label{triv-2}
R_{u,v}(\tau)-{1\over{2\pi i}}\tau(\log(uw^*v^*w))\in \rho_A(K_0(A)).
\eneq
\end{lem}

\begin{proof}
It suffices to show that there is one piece-wisely smooth and continuous
path $\{U(t): t\in [0,1]\}\in M_2(A)$ such that
$U(0)={\rm diag}(u, 1),$ $U(1)={\rm diag}(v,1)$ and
$$
{1\over{2\pi i}}\int_0^1 \tau(U(t)'U(t)^* dt={1\over{2\pi i}}\tau(\log(uw^*v^*w)).
$$
To see this, let $h={1\over{2\pi i}}\log(uw^*v^*w).$
Define $U(t)={\rm diag}(u\exp(i 4\pi ht), 1)$ for $t\in [0, 1/2].$
Define $U_1(t)=U(2t)$ for $t\in [0,1].$
Let $W={\rm diag}(w,w^*).$ Then $W=\prod_{j=1}^m\exp( i 2\pi h_j)$ for some
self-adjoint  elements $h_1,h_2,...,h_m\in M_2(A).$
Define $W(0)=1$ and
\beq\label{triv-3}
W(t)=(\prod_{j=1}^{k-1}\exp(i 2\pi h_j))\exp(i2\pi mh_kt)\tforal t\in (k-1/m, k/m],
\eneq
$k=1,2,...,m.$
Let $Z(t)=W(t)^*{\rm diag}(v,1)W(t)$ for $t\in [0,1].$
Then $Z(t)$ is a piece-wisely smooth and continuous path with $Z(0)={\rm diag}(v,1)$ and $Z(1)=W^*{\rm diag}(v,1)W.$
It is straightforward to compute that the de la Harpe-Skandalis  determinant
$$
Det(W(t))=0.
$$
Define $U(t)=Z(1-2t)$ for $t\in (1/2, 1]$ and define
$U_2(t)=Z(1-t)$ for $t\in [0,1].$
Now $U(t)$ is a continuous and piece-wisely continuous path with
$U(0)={\rm diag}(u,1)$ and $U(1)={\rm diag}(v,1).$
We then compute that
\beq\label{triv-4}
{1\over{2\pi}}\int_0^1 \tau({dU(t)\over{dt}} U(t)^*) dt&=& {\rm Det}(U(t))\\
&=&{\rm Det}(U_1(t))+{\rm Det}(U_2(t))\\
&=&{\rm Det}(U_1(t))+0\\
&=&{1\over{2\pi i}}\tau(\log(uw^*v^*w)).
\eneq
for all $\tau\in T(A).$

\end{proof}

\begin{lem}\label{Smrot}
Let $A$ be a unital separable \CA\, of stable rank one.
Suppose that $u, v\in  U(A)$ with
$uv^*\in CU(A).$ Then, for any $\dt>0,$ there exists $a\in A_{s.a.}$ with
$\|a\|<\dt$ such that
$$
\hat{a}-R_{u,v}\in \rho_A(K_0(A)).
$$
\end{lem}

\begin{proof}
This follows from the fact that $R_{u,v}\in \overline{\rho_A(K_0(A))}.$
\end{proof}

\begin{lem}\label{homt} {\rm (Theorem 6.2 of \cite{Lnhomt1})}
Let $\ep>0$ and let $\Delta: (0,1)\to (0,1)$ be a non-decreasing function. There exists $\dt>0$ and $\sigma>0$ satisfying the following:
For any unital separable simple \CA\, $A$ with $TR(A)\le 1$ and $u,\,v\in U(A)$ such that
\beq\label{homt-1}
\mu_{\tau\circ \phi}(I_a)\ge \Delta(a)\tforal \tau\in T(A)
\eneq
and for all arc $I_a$ with length at least $a\ge \sigma,$
where $\phi: C(\T)\to A$ is the \hm\, defined by $\phi(f)=f(u)$ for all $f\in C(\T),$
\beq\label{homt-2}
\|[u,\, v]\|<\dt,\,\,\, [v]=0\,\,\,{\rm in}\,\,\, K_1(A)\andeqn {\rm bott}_1(u,v)=0,
\eneq
there exists a continuous path of unitaries $\{v(t): t\in [0,1]\}\subset U_0(A)$ such that
\beq\label{homt-3}
\|[v(t), u]\|<\ep\tforal t\in [0,1], \,\,\, v(0)=v\andeqn v(1)=1.
\eneq
\end{lem}




The following is an  variation of a special case of 5.1 of \cite{LN-3}.

\begin{lem}\label{LN51}
Let $\ep>0$ and let $\Delta: (0,1)\to (0,1)$ be a non-decreasing function.
There is $\dt>0,$ $\eta>0,$ $\sigma>0$ and there is a finite subset
${\cal G}\subset C(\T)_{s.a.}$ satisfying the following:
For any unital separable simple \CA\, $A$ with $TR(A)\le 1,$ any pair of unitaries $u, v\in A$
$sp(u)=\T$
and
$[u]=[v]$ in $K_1(A),$
$$
\mu_{\tau\circ \phi}(I_a)\ge \Delta(a)\tforal \tau\in T(A)
$$
for all intervals $I_a$ with length at least $\eta,$ where
$\phi: C(\T)\to A$ is the \hm\, defined by $\phi(f)=f(u)$ for all $f\in C(\T),$
\beq\label{LN51-2}
|\tau(g(u))-\tau(g(v))|<\dt \tforal \tau\in T(A)
\eneq
and for all $g\in {\cal G},$
\beq\label{LN51-3}
uv^*\in CU(A),
\eneq
 for any $a\in \Aff(T(A))$ with $a-R_{u,v}\in \rho_A(K_0(A))$ and $\|a\|<\sigma$ and any $y\in K_1(A),$ there is a unitary
$w\in A$ such that $[w]=y,$ 
\beq\label{LN51-4}
\|u-w^*vw\|<\ep\andeqn\\
{1\over{2\pi i}}\tau(\log(u^*w^*vw))=a(\tau)\tforal \tau\in T(A).
\eneq
\end{lem}

\begin{proof}
Let $\ep>0$ and $\Delta$ be  given.
Choose $\ep>\theta>0$ such that,
$\log(u_1),$ $\log(u_2)$ and $\log(u_1u_2)$ are well defined and
\beq
\tau(\log(u_1u_2))=
\tau(\log(u_1))+ \tau(\log(u_2))
\eneq
for all $\tau\in T(A)$ and
for any unitaries
$u_1, u_2$
such  that
$$
\|u_j-1\|<\theta,\,\,\, j=1,2.
$$

Let $\dt'>0$ (in place of $\dt$) be required by \ref{botext} for $\theta/2$ (in place of $\ep$).
Put $\sigma=\dt'/2.$
Let $\dt>0$ and $\eta$ be required by \ref{QT} for $\min\{\sigma, \theta, 1\}$ (in place of $\ep$) and $\Delta.$
Suppose that $A$ is a unital separable simple \CA\, with $TR(A)\le 1$ and
$u,\, v\in U(A)$ satisfy the assumption for the above $\dt,$ $\eta$ and
$\sigma.$
Then, by \ref{QT}, there exists a unitary $z\in U(A)$ such that
\beq\label{LN51-5}
\|u-z^*vz\|<\min\{\theta, \sigma, 1\}.
\eneq
Let $b={1\over{2\pi i}}\log(u^*z^*vz).$ Then
$\|b\|<\min\{\theta, \sigma, 1\}.$  By \ref{triv},
$\hat{b}-R_{u,v}\in \rho_B(K_0(A)).$

Let $a\in \Aff(T(A))$ be such that
$\|a\|<\sigma$ and $a-R_{u,v}\in \rho_A(K_0(A))$ as given by the lemma. It follows that $a-\hat{b}\in \rho_A(K_0(A)).$ Moreover, $\|a-\hat{b}\|<2\sigma<\dt'.$
It follows from \ref{botext} that there exists a unitary $z_1\in A$ such that
\beq\label{LN15-6}
[z_1]=-y-[z],\,\,\,
\|[u, z_1]\|<\theta/2\andeqn {\rm bott}_1(u,z_1)(\tau)=a(\tau)-\tau(b)
\eneq
for all $\tau\in T(C).$

Define $w=zz_1^*.$ Then
\beq\label{LN51-7}
[w]=y\andeqn \|u-w^*vw\|<\theta<\ep.
\eneq

We compute that
\beq\label{LN51-8}
{1\over{2\pi i}}\tau(\log(u^*w^*vw))&=&
{1\over{2\pi i}}\tau(\log(u^*z_1z^*vzz_1^*))\\
&=&{1\over{2\pi i}}\tau(\log(u^*z_1uu^*z^*vzz_1^*))\\
&=& {1\over{2\pi i}}\tau(\log(z_1^*u^*z_1uu^*z^*vz))\\
&=&{1\over{2\pi i}}(\tau(\log(z_1u^*z_1^*u)+\tau(\log(u^*z^*vz)))\\
&=&{\rm bott}_1(u,z_1)(\tau)+\tau(b)\\
&=& a(\tau)\,\,\,\rforal \tau\in T(A),
\eneq
where we use the Exel's formula for bott element (see Lemma 3.5 of \cite{Lninv}) in the second last equality.
\end{proof}

\begin{lem}\label{Appu}
Let $\ep > 0$ and  let $\Delta: (0, 1) \to  (0, 1)$ be a non-decreasing map. There exists $\eta>0,\, \dt > 0$ and
a finite subset ${\cal G}\in C(\T)_{s.a.}$  satisfying
the following:
Suppose that $A$ is a ${\cal Z}$-stable unital separable simple \CA\,
 in ${\cal A}_1$
and suppose that $u, v\in U(A)$ are two unitaries such that ${\rm sp}(u)=\T,$
\beq\label{Appu-1}
\mu_{\tau\circ \phi}(I_a)\ge \Delta(a)
  \tforal \tau\in T(A)
 \eneq
 and for all arcs $I_a$ with length at least $a\ge \eta,$
 where $\phi: C(\T)\to A$ is defined by $\phi(f)=f(u)$ for all $f\in C(\T)$ and
 \beq\label{Appu-2}
 |\tau(g(u))-\tau(g(v))|<\dt\tforal a\in {\cal G}\tand
\tforal \tau\in T(A),\\\nonumber
{[u]}={[v]}\,\,\, {\rm in} \,\,\, K_1(A)\tand  uv^*\in CU(A).
\eneq
Then there exists a unitary  $w\in U( A )$ such that
\beq\label{Appu-3}
\|w^*uw-v\|<\ep.
\eneq
\end{lem}

\begin{proof}
We first note, by \cite{LN-2},
that $TR(A\otimes M_\fr)\le 1$ for any supernatural number $\fr.$
Let $\phi: C(\T)\to A$ be the monomorphism defined by
$\phi(f)=f(u).$
For any $a\in(0, 1)$, denote by
$$\Delta(a)=\inf\{\mu_{\tau\circ\psi}(O_a);\ \tau\in T(A), \textrm{$I_a$ an open arcs of length $a$ in $\T$}\}.$$ Since $A$ is simple, one has that $0<\Delta(a)\leq 1$ (for all $a\in (0,1)$) and $\Delta(a)\to 0$ as $a\to 0$. As in Proposition 11.1 of \cite{Lnappeqv},
\beq\label{Appu-n+1}
\mu_{\tau\circ \phi}(I_a)\ge \Delta(a)\tforal \tau\in T(A)
\eneq
and all arcs with length $a>0.$
Let $\ep>0.$

Let $\p$ and $\q$ be a pair of relatively prime supernatural numbers of infinite type with $\mathbb Q_\p+\mathbb Q_\q=\mathbb Q$. Denote by $M_\p$ and $M_\q$ the UHF-algebras associated to $\p$ and $\q$ respectively.
Let $\imath_\fr: A\to A\otimes M_\fr$ be the embedding defined by
$\imath_\fr(a)=a\otimes 1$ for all $a\in A,$ where $\fr$ is a supernatural number.
Define $u_\fr=\imath_\fr(u)$ and $v_\fr=\imath_\fr(v).$ Denote
by $\phi_{\fr}: C(\T)\to A\otimes M_{\fr}$ the \hm\, given by
$\phi_\fr(f)=f(u_\fr)$ for all $f\in C(\T).$

For any supernatural number $\fr=\p, \q$, the C*-algebra
$A\otimes M_\fr$ has tracial rank at most one.

Let $\dt_1>0$ (in place of $\dt$)  and $d_1>0$ (in place of $\sigma$) be required by \ref{homt}
for $\ep/6.$

Without loss of generality, we may assume that $\delta_1<\ep/12$ and is small enough and $\mathcal G$ is large enough so that
$\mathrm{bott}_1(u_1, z_j)$  and
$\mathrm{bott}_1(u_1, w_j)$ are well defined and
\beq\label{Appu-10}
\mathrm{bott}_1(u_1,w_j)=
\mathrm{bott}_1(u_1, z_1)+\cdots + \mathrm{bott}_1(u_1, z_j)
\eneq
if $u_1$ is a unitary and  $z_j$ is any unitaries with $\|[u_1, z_j]\|<\delta_1,$  where $w_j=z_1\cdots z_j,$ $j=1,2,3,4.$

Let $\dt_2>0$ (in place $\dt$) be require by \ref{botext} for $\dt_1/8$
(in place of $\ep$).

Furthermore, one may assume that $\delta_2$ is sufficiently small such that for any unitaries $z_1, z_2$ in a C*-algebra with tracial states, $\tau(\frac{1}{2\pi i}\log(z_iz_j^*))$ ($i, j=1,2,3$) is well defined and
\vspace{-0.1in}
$$
\tau(\frac{1}{2\pi i}\log(z_1z_2^*))=\tau(\frac{1}{2\pi i}\log(z_1z_3^*))+ \tau(\frac{1}{2\pi i}\log(z_3z_2^*))
$$
for any tracial state $\tau$, whenever $\|z_1-z_3\|<\delta_2$ and $\|z_2-z_3\|<\delta_2$.
We may further assume that $\dt_2<\min\{\dt_1, \ep/6,1\}.$

Let $\dt>0,$ $d_2>0$ (in place of $\eta$) and
$\dt_3>0$ (in place of $\sigma$) required by \ref{LN51} for $\dt_2$ (in place of $\ep$).

Let $\eta=\min\{d_1, d_2\}.$

Now assume that $u$ and $v$ are two unitaries which satisfy the assumption of the lemma with above $\dt$ and $\eta.$

Since $uv^*\in CU(A),$ $R_{u,v}\in \overline{\rho_A(K_0(A))}.$
It follows that there is $a\in \aff(T(A))$ with $\|a\|<\dt_3/2$ such that
$a-R_{u,v}\in \rho_A(K_0(A)).$ Then the image of $ a_\p-R_{u_\p,v_\p}$
is in $\rho_{A\otimes M_\p}(K_0(A\otimes M_\p))),$
 where $a_\p$ is the image of $a$ under the map
 induced by $\imath_\p.$ The same holds for $\q$.
Note that
\beq\label{Appu-19-1}
\mu_{(\tau\otimes t)\circ \phi_\fr}(I_a)\ge \Delta(a)\tforal \tau\in T(A)
\eneq
where $t$ is the unique tracial state on $M_\fr,$ and for all $a>0,$ $\fr=\p,\,\q.$
  By Lemma \ref{LN51} there {exist} unitaries $z_\p\in A\otimes M_\p$ and $z_\q\in A\otimes M_\q$ such that
 $$ [z_\fr]=0,\,\,\, {\rm in}\,\,\, K_1(A\otimes M_\fr),\,\,\, \fr=\p,\,\q,\,\,\,\|u_\p-z_\p^* v_\p z_\p\|<\delta_2
{\andeqn}
\|u_\q-z_\q^* v_\q z_\q\|<\delta_2.
$$
\vspace{-0.2in}
Moreover,
\beq\label{Appu19}
\tau(\frac{1}{2\pi i}\log(u_\p^* z_\p^* v_\p z_\p))
=a_\p(\tau)\tforal \tau\in \mathrm{T}(A_\p)
\andeqn \\
\tau(\frac{1}{2\pi i}\log(u_\q^* z_\q^* v_\q z_\q))
=a_\q(\tau)\tforal \tau\in\mathrm{T}(A_\q).
\eneq
We then identify $u_\p, u_q$ with $u\otimes 1$ and
$z_\p$ and $z_\q$ with the elements in $A\otimes M_\p\otimes M_\q=A\otimes Q.$ In the following computation, we also identify
$T(A)$ with $T(A_\p),$ $T(A_\q),$ and
$T(A_\p),$ or $T(A_\q)$ with $T(A\otimes Q)$
by identify $\tau$ with $\tau\otimes t,$ where $t$ is the unique tracial state on $M_\p,$ or $M_\q,$ or $Q.$
In particular,
\beq\label{Appu-20}
a_\p(\tau\otimes t)&=&\tau(a)\tforal \tau\in T(A)\andeqn\\
a_\q(\tau\otimes t)&=&\tau(a) \tforal \tau\in T(A)
\eneq

We compute that
by the Exel formula (see Lemma 3.5 of \cite{Lninv} ),
\begin{eqnarray}
(\tau \otimes t)(\mathrm{bott}_1(u\otimes 1, z_\p^* z_\q))&=&
 (\tau\otimes t)(\frac{1}{2\pi i}\log(z_\p^* z_\q(u^*\otimes 1)z_\q^* z_\p(u\otimes 1)))\\
&=&{(\tau\otimes t)(\frac{1}{2\pi i}\log(z_\q(u^*\otimes 1) z_\q^* z_\p(u\otimes 1)z_\p^*)}\\
&=&{(\tau\otimes t)(\frac{1}{2\pi i}\log(z_\q(u^*\otimes 1) z_\q ^*(v\otimes 1))}\\
&&+(\tau\otimes 1)({1\over{2\pi i}}\log((v^*\otimes 1)z_\p(u\otimes 1)
z_\p^*)\\
&=& (\tau\otimes t)(\frac{1}{2\pi i}\log(u_\q^* z_\q^* v_\q z_\q))\\
&&+(\tau\otimes t)({1\over{2\pi i}}\log(v_\p^* z_\p u_\p z_\p^*))\\
&=&\tau(a) -\tau(a)=0
\end{eqnarray}
for all $\tau\in T(A).$ It follows that
\beq\label{Appu-21}
\tau(\mathrm{bott}_1(u\otimes 1, z_\p^* z_\q))=0.
\eneq
for all $\tau\in T(A\otimes Q).$



Let $y=\mathrm{bott}_1(u\otimes 1, z_\p^* z_\q)\in \ker \rho_{A\otimes Q}.$  Since $\Q,$ $\Q_\p$ and $\Q_\q$ are flat $\Z$ modules, as in
 the proof of 5.3 of \cite{LN-3},
\beq\label{ker}
{\rm ker}\rho_{A\otimes Q}&=&{\rm ker}\rho_A\otimes \Q\\\label{ker-1}
{\rm ker}\rho_{A\otimes M_\fr}&=&{\rm ker} \rho_A\otimes \Q_\fr,\,\,\,\,\fr=\p\andeqn \fr=\q.
\eneq
It follows that there are $x_1, x_2,...,x_l\in \rho_A(K_0(A))$
and $r_1,r_2,...,r_l\in \Q$ such that
$$
y=\sum_{j=1}^l x_j\otimes r_j.
$$
Since $\Q=\Q_\p+Q_\q,$ one has $r_{j,\p}\in \Q_\p$ and
$r_{j, \q}\in \Q_\q$ such that
$r_j=r_{j,\p}-r_{j,\q}.$
So
$$
y=\sum_{j=1}^l x_j\otimes r_{j, \p}-\sum_{j=1}^l x_j\otimes r_{j, \q}.
$$
Put $y_\p=\sum_{j=1}^l x_j\otimes r_{j, \p}$ and $y_\q=\sum_{j=1}^l x_j\otimes r_{j, \q}.$
 Then, by (\ref{ker-1}),  $y_\p\in {\rm ker}\rho_{A\otimes M_\p}$ and
$y_\q\in {\rm ker}\rho_{A\otimes M_\q}.$
It follows from \ref{botext} that there are unitaries $w_\p\in A\otimes M_\p$ and $w_\q\in A\otimes M_\q$ such that
$[w_\fr]=0$ in $K_1(A\otimes M_\fr)$ ($\fr=\p,\q$),
\beq\label{Appu-22}
\|[u_\p\,, w_\p]\|<\dt_1/8,\,\,\, \|[v_\q,\, w_\q]\|<\dt_1/8,\\
{\rm bott}_1(u_\p, w_\p)=y_p\andeqn
{\rm bott}_1(u_\q, w_q)=y_\q.
\eneq

Put $W_\p=z_\p w_\p\in A\otimes M_\p$ and $W_\q=z_\q w_\q\in A\otimes M_\q.$ Then
\beq\label{Appu-23}
\|u_\p-W_\p^* v_\p W_\p\|<\dt_2+\dt_1/8<\ep/6\andeqn
\|u_\q-W_\q^* v_\q W_\q\|<\dt_2+\dt_1/8<\ep/6.
\eneq
Note, again,  that $u_\fr=u\otimes 1$ and $v_\fr=v\otimes 1,$  $\fr=\p,\,\q.$
With identification of $W_\fr, w_\fr, z_\fr$ with unitaries in $A\otimes Q,$
we also have
\beq\label{Appu-24}
\|[ u\otimes 1, \, W_\p^* W_\q]\|<\dt_1/4
\eneq
and
\beq\label{Appu-25}
{\rm bott}_1(u\otimes 1, W_\p^* W_\q) &=&
\mathrm{bott}_1(u\otimes 1, w_\p^*z^*_\p  z_\q  w_\q  )\\
&=&
 \mathrm{bott}_1(u\otimes 1, w^*_\p)+
 \mathrm{bott}_1(u\otimes 1, z^*_\p z_\q)+
 \mathrm{bott}_1(u\otimes 1, w_\q)\\
 &=& -y_\p +(y_\p-y_\q)+y_\q=0.
 \eneq
Let $Z_0=W_\p ^*W_\q.$ Then $Z_0\in U_0(A)$ (since $TR(A)\le 1$).  Then it follows from  the choice of $\dt_1,$ (\ref{Appu-19-1}) and \ref{homt} that there is a continuous path of unitaries $\{Z(t): t\in [0,1]\}\subset A\otimes Q$ such that
$Z(0)=Z_0$ and $Z(1)=1$ and
\beq\label{Appu-26}
\|[u\otimes 1,\, Z(t)]\|<\ep/6\tforal t\in [0,1].
\eneq
Define $U(t)=W_\p Z(t).$ Then $U(0)=W_\p$ and $U(1)=W_\q.$ So, in particular, $U(0)\in A\otimes M_\p$ and $U(1)\in A\otimes M_\q.$
So, $U\in A\otimes {\cal Z}_{\p, \q}\subset A\otimes {\cal Z}$ is a unitary and, by  (\ref{Appu-23}) and  (\ref{Appu-26}),
\beq\label{Appu-27}
\|u\otimes 1-U^*(v\otimes 1)U\|<\ep/3.
\eneq

Note that we assume that $A\otimes {\cal Z}\cong A.$ Let $\imath: A\to A\otimes {\cal Z}$ be the embedding defined by $\imath(a)=a\otimes 1$ for all $a\in A$ and $j: A\otimes {\cal Z}\to A$ such that $j\circ \imath$ is approximately inner.
Let $V\in A$ be a unitary such that
\beq\label{Appu-28}
\|c-V^*j\circ \imath(c)V\|<\ep/3\tforal c\in \{u, v\}.
\eneq
Then,  let $w=Vj(U)V^*\in U(A).$
\beq\label{Appu-29}
\hspace{-0.8in}\|u- w^*uw\| &\le & \|u-V^*j(u\otimes 1)V\|
+\|V^*j(u\otimes 1)V- V^*j(U)^*j(v\otimes 1)j(U)V\|\\
&& +\|V^*j(U)^*j(v\otimes 1)j(U)V-V^*j(U)^*VvV^*j(U)V\|\\
&<& \ep/3+\|u\otimes 1-U^*(v\otimes 1)U\|+\|j\circ \imath(v)-V^*vV\|\\
&<& \ep/3+\ep/3+\ep/3=\ep.
\eneq

\end{proof}

\begin{lem}\label{quasitrace}
Let $A$ be a unital separable simple \CA\, in ${\cal A}_1.$
Then every quasi-trace on $A$ extends to a trace. Moreover,
if in addition, $A$ is ${\cal Z}$-stable, then
$$
W(A)=V(A)\sqcup {\rm LAff}(T(A)),
$$
where $W(A)$ is the Cuntz semi-group of $A,$ $V(A)$ is the equivalence classes of projections in $\cup_{n=1}^{\infty} M_n(A)$ and
$LAff(T(A))$ is the set of all bounded real lower-semi-continuous affine functions on $T(A).$

\end{lem}

\begin{proof}
Note that $TR(A\otimes Q)\le 1.$ Therefore every quasi-trace on $A\otimes Q$ is a trace. Suppose that $s$ is a quasi-trace on $A,$
then $s\otimes t$ is a trace on $A,$ where $t$ is the unique tracial state of $Q.$ Therefore $s\otimes t$ on $A\otimes \C 1_{Q}$ is a trace.
This implies that $s$ is a trace.

 The second part of the statement follows from (the proof of ) Cor. 5.7 of \cite{BPT}.
 In Cor. 5.7 of \cite{BPT}, $A$ is assumed to be also exact. But that was only used
 so that every quasi-trace is a trace.

\end{proof}

\begin{lem}\label{embedding}
Let $A\in {\cal A}_1$ be a unital separable simple ${\cal Z}$-stable \CA.
Let $\Gamma: C([0,1])_{s.a.}\to \Aff(T(A))$ be a continuous affine map with
$\Gamma(1)(\tau)=1$ for all $\tau\in T(A)$ for some $a\in A_+$ with $\|a\|\le 1.$
Then there exists a unital monomorphism $\phi: C([0,1])\to  A$ such that
$$
\tau(\phi(f))=\Gamma(f)(\tau)\tforal \tau\in T(A)
$$
and $f\in C([0,1]).$
\end{lem}

\begin{proof}
Let $\p$ and $\q$ be relatively prime supernatural numbers with $\Q_\p+\Q_\q=\Q.$
Let $M_\fr$ be the UHF-algebra associated with the supernatural number $\fr,$ $\fr=\p,\q.$ Let $Q$ be the UHF-algebra
such that $(K_0(Q), (K_0(Q)_+, [1_\Q])=(\Q, \Q_+, 1).$   By the assumption $TR(A\otimes M_\fr)\le 1$ and $TR(A\otimes Q)\le 1.$
It follows from 8.4 of \cite{Lninv} that are $h_\fr\in (A\otimes M_\fr)_{s.a.}$ such that ${\rm sp}(h)=[0,1]$
and $(\tau\otimes t)\circ \phi_\fr(f)=\Gamma(f)(\tau)$ for all $\tau\in T(A)$ and $f\in C([0,1])_{s.a.},$ where
$t$ is the unique tracial state on $M_\fr,$ $\fr=\p,\,\q.$
We use the same notation for $\phi_\fr$ for the unital monomorphisms $C([0,1])\to A\otimes M_{\fr}\to A\otimes Q$ composed by
$\phi_\fr$ and the embedding from $A\otimes M_\fr\to A\otimes Q.$
Note that $K_0(C([0,1]))=\Z$ and $K_1(C([0,1]))=\{0\}.$
Then $[\phi_\p]=[\phi_\q]$ in $KK(C([0,1]), A\otimes Q)$ and,   $\phi_\p$ and $\phi_\q$ induce the  same map from $T(A\otimes Q)$ into $T(C([0,1]))$  as well as
the same map from  $U(C([0,1])/CU(C([0,1])$ into $CU(A\otimes Q)/CU(A\otimes Q).$
Moreover since $K_1(C([0,1]))=\{0\}.$ They induce zero rotation map.
It follows from 10.7 of \cite{Lninv} that $\phi_\p$ and $\phi_\q$ are strongly asymptotically unitarily equivalent, i.e.,
there exists a continuous path of unitaries $\{u(t): t\in [0, 1)\}\subset A\otimes Q$ such that
$$
\lim_{t\to \infty} u(t)^*\phi_\p(f) u(t)=\phi_\q(f)\tforal f\in C([0,1]).
$$
Define $\psi: C([0,1])\to A\otimes Z_{\p,\q}$ by
$$
\psi(f)(t)=u(t)^*\phi_\p(f)u(t)\tforal t\in [0,1)\andeqn \psi(f)(1)=\phi_\q(f)\tforal f\in C([0,1]).
$$
Note $\psi(f)(0)=\phi_\p(f)\in A\otimes M_\p$ and $\psi(f)(1)=\phi_q(f)\in A\otimes M_\q$ for all $f\in C([0,1]).$
By embedding $A\otimes {\cal Z}_{\p,\q} $ into $A\otimes {\cal Z},$ we obtain a unital monomorphism $\phi: C([0,1])\to A\otimes {\cal Z}\cong A.$
It is easy to check that so defined $\phi$ meets the requirements.

\end{proof}

\begin{NN}\label{0pimeas}
{\rm
Let $A$ be a unital simple \CA\, with $T(A)\not=\emptyset.$
Let $u\in U(A)$ be a unitary with ${\rm sp}(u)=\T.$
For each $\tau,$ let $\mu_{\tau}$ be the Borel probability measure
on $\T$ induced by  state $\tau\circ f(u)$ (for all $f\in C(\T)$) on $\T.$
Fix $n\ge 1,$ let $\log: \{e^{it}: t\in [-\pi+\pi/n, \pi]\}\to[-\pi+\pi/n, \pi]$
be the usual logarithm  map.
Consider the measure $\nu_{\tau, n}$ on $(-\pi, \pi]$
defined by
$$
\nu_{\tau,n}(E)=\mu_{\tau}(\{e^{it}: t\in E\cap [-\pi+\pi/n, \pi]\})
$$
for all Borel sets $E\subset (-\pi, \pi].$
Define
$$
\nu_{\tau}(E)=\lim_{n\to\infty}\mu_{\tau, n}(E)
$$
for all Borel sets $E\subset (-\pi, \pi].$ It is easy to check
that $\nu_{\tau}$ is a measure on $(-\pi,\pi].$
Let $f\in C_0((-\pi, \pi])_{s.a.}$ defined
$$
\Gamma(f)(\tau)=\int_{(-\pi,\pi]} f d\nu_{\tau}.
$$
Note that
$$
\Gamma(f)(\tau)=\lim_{n\to\infty}\int_{(-\pi+\pi/n, \pi]} f\circ \log d\mu_{\tau}.
$$

Let $g_n(t)=1$ if $t\in [-\pi+\pi/n, \pi],$ $g_n(t)=0$ if $t\in [-\pi,-\pi+\pi/2n]$ and $g(t)$ is linear in $(-\pi+\pi/2n, -\pi+\pi/n).$ Note that $0\le g_n\le 1$ and
$g_n\in C(\T)_+.$
It is clear that $\Gamma(g_n)\in \Aff(T(A))$ and
$\Gamma(g_n)\le \Gamma(g_{n+1})$ and
$\Gamma(g_n)(\tau)\to 1$ for each $\tau\in \Aff(T(A)).$
It follows from the Dini theorem that $\Gamma(g_n)$ converges to $1$ uniformly on $T(A).$
On the other hand
\beq\label{erank-n}
|\int_{(-\pi,\pi)} f(1-g_n)d\nu_{\tau}| &\le & \int_{(-\pi, \pi]}|f|^2 d\nu_\tau \int_{(-\pi, \pi]}
(1-g_n)^2d\nu_{\tau}\\
&\le & \int_{(-\pi, \pi]}|f|^2 d\nu_\tau\int_{(-\pi, \pi]}
(1-g_n)d\nu_{\tau}\to 0
\eneq
uniformly on $T(A).$
This implies that $\Gamma(f)$ is continuous on $T(A).$
If $g\in C([-\pi ,\pi])_{s.a.},$ we may write
$g(t)=g(0)+(g(t)-g(0)).$  Define
$\Gamma(g)=g(0)+\Gamma(g-g(0)).$ This provides
an affine continuous map from $C([-\pi ,\pi ])_{s.a.}$ to $ \Aff(T(A)).$

We check that
$$
\tau(f(u))=\int_{(-\pi ,\pi]} f\circ \exp(it) d\nu_\tau(t)=\Gamma(f\circ \exp(it))(\tau) \tforal f\in C(\T)_{s.a.}.
$$

 In the above, we can replace $-\pi$ by $0$ and $\pi$ by $2\pi.$

We will keep this notation in the next proof.

}

\end{NN}

\begin{thm}\label{erank}
Let $A\in {\cal A}_1$ be a unital separable simple ${\cal Z}$-stable \CA.
Let $u\in U_0(A)$ be a unitary. Then, for any $\ep>0,$ there exists
a self-adjoint  element $h\in A$ such that
\beq\label{erank-1}
\|u-\exp(ih)\|<\ep.
\eneq
In the other words
$$
{\rm cer}(A)\le 1+\ep.
$$
\end{thm}

\begin{proof}
Let $u\in U_0(A).$ If ${\rm sp}(u)\not=\T,$ then $u$ is an exponential.
So we may assume that ${\rm sp}(u)=\T.$
Let $\ep>0.$ Let $\phi: C(\T)\to A$ be defined by $\phi(f)=f(u)$ for all $f\in C(\T).$ It is a unital monomorphism.
It follows from Proposition 11.1 of \cite{Lnappeqv} that there is a non-decreasing function
$\Delta_1: (0,1)\to (0,1)$ such that
\beq\label{exrank-1}
\mu_{\tau}(O_a)\ge \Delta_1 (a)\tforal \tau\in T(A)
\eneq
for all arcs $I_a$ of $\T$ with length $a\in (0,1).$
Define $\Delta=(1/2)\Delta_1.$

Let $\eta>0,$ $\dt>0$ and let ${\cal G}\subset C(\T)$ be a finite subset
required by \ref{Appu} for $\ep/2$ (in place of $\ep$). Without loss of generality, we may assume that $\|g\|\le 1$ for all $g\in {\cal G}.$
Let $\sigma=\min\{\eta/2, \dt/2\}.$

Let $\Gamma: C([0, 2\pi])_{s.a.}\to \Aff(T(A))$ be the map defined in
\ref{0pimeas} (using $[0, 2\pi]$ instead of $[-\pi,\pi]$).
Define $\Gamma_1: C([0,2\pi])_{s.a.}\to \Aff(T(A))$  as follows: define
$$
\Gamma_1(f)(\tau)=(1-\sigma)\Gamma(f)(\tau)\tforal \tau\in T(A)
$$
and for all $f\in C_0((0, 2\pi])$ and define
$$
\Gamma_1(f)=f(0)+\Gamma_1(f-f(0))(\tau)\tforal
\tau\in T(A)
$$
and for all $f\in C([0,2\pi])_{s.a.}.$
It follows that, for any $f\in C(\T)_{s.a.}$ with $\|f\|\le 1,$
\beq\label{erank-n}
|\tau(f)-\Gamma_1(f\circ \exp)|=|\Gamma(f\circ \exp)-\Gamma_1(f\circ \exp)|<\sigma\tforal \tau\in T(A),
\eneq
where $\exp: [0, 2\pi]\to \T$ is defined by $\exp(t)=e^{it}$ for all $t\in [0,2\pi].$
Note since $A$ is simple and ${\rm sp}(u)=\T,$
$\Gamma_1$ is strictly positive.
It follows \ref{embedding}  that there is a self-adjoint  element
$b\in A$ such that ${\rm sp}(b)=[0, 2\pi]$ and
$\tau(f(b))=\Gamma_1(f)(\tau)$ for all $f\in C_0((0,2\pi]).$
It follows that
\beq\label{erank-1}
d_{\tau}(b)=\lim_{n\to\infty}\tau(b^{1/n})\le (1-\sigma)\tforal \tau\in T(A).
\eneq

Note that since $A$ is also ${\cal Z}$-stable, by \ref{quasitrace},
$W(A)=V(A)\sqcup {\rm LAff}(T(A)).$
There  are mutually orthogonal  elements $ a_1, c_1, c_2 \in M_K(A)_+$  with $0\le a_1, c_1, c_2\le 1$ for some integer $K\ge 1$ such
that
\beq\label{erank-2}
d_{\tau}(a_1)=1-\sigma/2, d_{\tau}(c_1)=d_{\tau}(c_2)=\sigma/5\tforal \tau\in T(A).
\eneq
Put $a_2=a_1+c_1+c_2.$ Note
that $0\le  a_2\le 1$ and
\beq\label{erank-2+n}
d_{\tau}(a_2)=1-{9\sigma\over{10}}<1\tforal \tau\in T(A).
\eneq
By the strict comparison, (\ref{erank-1}), (\ref{erank-2+n})   and the fact that $A$ has stable rank one, we may assume, without loss of generality,  that
$$
a_2\in A\andeqn b\in \overline{a_1M_K(A)a_1}.
$$
Suppose that
\beq\label{erank2+}
{\rm Det}(u)(\tau)=s(\tau)\tforal \tau\in T(A)
\eneq
for some $s\in \Aff(T(A)).$

The above argument also shows that there are $b_1\in \overline{c_1Ac_1}$ and $b_2\in \overline{c_2Ac_2}$ such
that
\beq\label{erank-3}
\tau(b_1)=\sigma s(\tau)/6\andeqn \tau(b_2)=\sigma\tau(b)/6\tforal \tau\in T(A).
\eneq
Let $$h_1={-6b_2\over{\sigma}}+{6b_1\over{\sigma}}+b.$$
Note that
\beq\label{erank-3+1}
\tau(h_1)=(6/\sigma)\tau(b_1)=s(\tau)\tforal \tau\in T(A).
\eneq
Define
$v=\exp(ih_1).$  One checks, by (\ref{erank-3+1}),  that
\beq\label{erank-4}
{\rm Det}(v)={\rm Det}(u)
\eneq
Therefore
\beq\label{erank-5}
uv^*\in CU(A).
\eneq
Let ${\rm sp}(h_1)\subset [-m_1\pi, m_2\pi]$ for some integers $m_1, m_2\ge 0.$
Note that $f(b_j)\in \overline{c_jAc_j}$ if $f\in C({\rm sp}(b_j)),$  $j=1,2.$ So, if, in addition,
$\|f\|\le 1,$  by (\ref{erank-2}),
\beq\label{erank-5+1-1}
|\tau(f(b_j))|<\sigma/5\tforal \tau\in T(A).
\eneq
We have, for any $f\in C([-m_1\pi, m_2\pi])_{s.a.}$ with $\|f\|\le 1,$  by (\ref{erank-5+1-1}),
\beq\label{erank-5+1}
|\tau(f(h_1))-\Gamma_1(f|_{[0, 2\pi]})|&=&|(\tau(f(b))+\tau(f(b_1))+\tau(f(b_2))-\Gamma_1(f|_{[0, 2\pi]})|\\\label{erank-5+2}
&= & \sigma/5+\sigma/5+|\tau(f(b))-\Gamma_1(f_{[0,2\pi]})|=2\sigma/5
\eneq
for all $\tau\in T(A).$
Therefore,  by (\ref{erank-n}),   (\ref{erank-5+1}) and (\ref{erank-5+2}), that
\beq\label{erank-6}
|\tau(g(v))-\tau(g(u))|<\sigma+2\sigma/5<\dt \tforal g\in {\cal G}
\eneq
It follows from \ref{Appu} that there exits a unitary $w\in U(A)$ such that
$$
\|u-w^*vw\|<\ep.
$$
Let  $h=w^*h_1w.$
Then
$$
\|u-\exp(ih)\|<\ep.
$$

\end{proof}

\begin{cor}\label{Zexpl}
Let ${\cal Z}$ be the Jiang-Su algebra.
Then
$$
{\rm cer}({\cal Z})=1+\ep.
$$
\end{cor}

We will prove much stronger result than the above for ${\cal Z}$ (see \ref{NNZlT}).

\section{Exponential length in $CU(A)$}


The following is known (something similar could be found in
\cite{Thoms-2} and \cite{Ph-fu}).
We state here for the convenience.

\begin{lem}\label{CET}
Let $u$ be a unitary in $C([0,1], M_n).$
Then, for any $\ep>0,$ there exist continuous functions
$h_j\in C([0,1])_{s.a.}$ such that
$$
\|u-u_1\|<\ep,
$$
where $u_1=\exp(i \pi H),$ $H=\sum_{j=1}^n h_jp_j$ and $\{p_1, p_2,...,p_n\}$ is a set of mutually orthogonal rank one projections in $C([0,1], M_n),$
and $\exp(i \pi h_j(t))\not=\exp(i\pi h_k(t))$ if $j\not=k$ for all
$t\in [0,1].$
Moreover,  suppose that $u(0)=\sum_{j=1}^n \exp(i a_j)p_j(0)$ for some real number $a_j$ which are distinct,
we may assume that $h_j(0)=a_j.$

 Furthermore,
if ${\rm det}(u(t))=1$ for all $t\in [0,1],$ then
we may also assume that ${\rm det}(u_1(t))=1$ for all $t\in [0,1].$
\end{lem}

\begin{proof}
The last part of the statement follows from Lemma 2.5 of \cite{Ph-fu}.
By Lemma 2.5, if ${\rm det}(u(t))=1$ for all $t\in [0,1],$ then
we can choose $u_1$ such that
$\|u-u_1\|<\ep$ and ${\rm det}(u_1(t))=1$ for all $t\in [0,1]$ and
$u(t)$ has distinct eigenvalues.
Therefore $u_1=\sum_{j=1}^n z_j(t)p_j(t),$
where $p_j(t)\in C([0,1], M_n)$ is a rank one projection,
$\sum_{j=1}^kp_j=1$ and
$z_j(t)\in C([0,1])$ with $|z_j(t)|=1$ for all $t\in [0,1]$
(This is standard, see, for example, the proof of 2.6 of \cite{Ph-fu}).
Let $z_j(t)=e^{ia(0)}$ for some real number $a(0).$ But
$z_j(t)=e^{ib_j(t)}$ for some real $b_j\in C([0,1]),$ $j=1,2,...,n.$
Note that $a_j(0)-b_j(0)=2k\pi$ for some integer $k.$
By replacing $b_j$ by $a_j(t)=h_j(t)+(a_j(0)-h_j(0)).$ Then
$z_j(t)=e^{ia_j(t)}$ and $z_j(0)=a_j(0),$ $j=1,2,...,n.$
In particular,
$$
u_1(t)=\sum_{j=1}^n e^{ia_j(t)}p_j(t)\tforal t\in [0,1].
$$

\end{proof}

%
%

\begin{lem}\label{Det=1} {\rm  (cf. Section 3 of \cite{GLN})}
Let $u\in C([0,1], M_n)$ be a unitary with ${\rm det}(u)(t)=1$ for each $t\in [0,1].$
Then, for any $\ep>0,$ there exists a self-adjoint  element $h\in C([0,1], M_n)$ such that
$\|h\|\le 1,$ $\tau(h)=0$ for each $\tau\in T(C([0,1], M_n)$  and
$$
\|u-\exp(i2\pi h)\| <\ep.
$$
In particular ${\rm length}(u)\le 2\pi.$
\end{lem}

\begin{proof}
The proof is taken from the section 3 of \cite{GLN}.
First, by \ref{CET}, without loss of generality, we may assume that $u(0)$ has distinct eigenvalues.
Suppose that
$$
u(0)=\sum_{j=1}^n \exp(i2\pi b_j)p_j(0),
$$
where $b_j\in (-1/2,1/2],$ $j=1,2,...,n.$

Then $\sum_{j=1}^n b_j=k$ for some integer $k.$
Since $b_j\in (-1/2, 1/2],$  $k\le n.$
Keep in mind that $b_j$ are distinct.
If $k\ge 1,$ to simplify notation, we may assume
that $b_j>0,$ $j=1,2,...,k,$ $b_{k+l}<b_{k}< b_j$
for $j< k$ and $l>0.$
Define $a_j=b_j-1,$ $j=1,2,...,k$ and $a_j=b_j,$ $j>k.$
Then
\beq\label{nN-1}
\sum_{j=1}^n a_j=0 \andeqn |a_j|<1.
\eneq
Note that $\max_ja_j< b_{k}.$ Since $b_j>-1/2,$ $\min_ja_j=b_k-1.$
Therefore, we also have
\beq\label{nN-2}
\max_j a_j-\min_ja_j<1.
\eneq
If $k<-1,$  we may assume that $b_j<0,$ $j=1,2,...,k,$
$b_{k+l}\ge b_k>b_j$ for $j\le k$ and $l>0.$ Define $a_j=b_j+1,$ $j=1,2,...,k$ and $a_j=b_j$ if $j>k.$
Then (\ref{nN-1}) and (\ref{nN-2}) also hold in this case.

By \ref{CET}, we may assume, without loss of generality, that
\beq\label{Det=1-1}
u(t)=\sum_{j=1}^n \exp(i \pi h_j(t))p_j(t),
\eneq
where $h_j(t)\in C([0,1])_{s.a.}$ and $\{p_1,p_2,...,p_n\}$ is a set of mutually orthogonal rank one projections. Moreover, we may assume
that ${\rm det}(u(t))=1$ for all $t\in [0,1]$ and
$u(t)$ has distinct eigenvalues at each point $t\in [0,1].$
 Furthermore, by \ref{CET}, we may also assume
that $h_j(0)=a_j,$ $j=1,2,...,n.$
We also have that $|h_j(0)|<1,$
\beq\label{det=1-3}
\sum_{j=1}^n h_j(0)=0\andeqn \max_jh_j(0)-\min_jh_j(0)<1.
\eneq
Since ${\rm det}(u(t))=1$ for all $t\in [0,1],$
\beq\label{det=1-4}
\sum_{j=1}^n h_j(t)\in \Z \rforal t\in [0,1].
\eneq
Since $\sum_{j=1}^nh_j(t)\in C([0,1]),$ it follows that it is a constant. By (\ref{det=1-3}),
\beq\label{det=1-5}
\sum_{j=1}^n h_j(t)=0\tforal t\in [0,1].
\eneq
Since  $u(t)$ has distinct eigenvalues, $h_j(t)-h_k(t)\not\in \Z,$ for any $t\in [0,1]$ when $j\not=k. $  We also have
$\max_j h_j(t)-\min_j h_j(t)$ is a continuous function. It follows from (\ref{det=1-3}) that
\beq\label{det=1-7}
0<\max_j h_j(t)-\min_j h_j(t)<1\rforal t\in [0,1].
\eneq
Now by (\ref{det=1-5}), either $h_j(t)=0$ for all $j,$ which is not possible, since $u(t)$ has $n$ distinct eigenvalues,
or, for some $j,$ $h_j(t)<0$  and for some other $j',$ $h_{j'}>0,$  it follows from (\ref{det=1-7}) that
\beq\label{det=1-8}
|h_j(t)|<1\tforal t\in [0,1].
\eneq
Now let $h=\sum_{j=1}^n h_j\in C([0,1], M_n)_{s.a.}.$ Then
\beq\label{det=1-9}
\|h\|<1,\tau(h)=0\tforal \tau\in T(A)\andeqn u=\exp(i2\pi h).
\eneq

\end{proof}

We will use the following theorem (Theorem 10.8 of \cite{Lnappeqv}).

\begin{thm}\label{QT}
Let $\ep > 0$ and  let $\Delta: (0, 1) \to  (0, 1)$ be a non-decreasing map. There exists $\sigma>0,\, \dt > 0$ and
a finite subset ${\cal G}\in C(\T)_{s.a.}$  satisfying
the following:
Suppose that $A$ is a unital separable simple \CA\, with tracial rank no more than one
and suppose that $u, v\in U(A)$ are two unitaries such that
\beq\label{QT-1}
\mu_{\tau\circ \phi}(I_a)\ge \Delta(a)
 \tforal  \tau\in T(A),
 \eneq
 and for all intervals $I_a$ with length at least $\sigma,$
 where $\phi: C(\T)\to A$ is defined by $\phi(f)=f(u)$ for all $f\in C(\T),$
 \beq\label{QT-2}
&& |\tau(g(u))-\tau(g(v))|<\dt\tforal a\in {\cal G},
\tforal \tau\in T(A),\\\nonumber
&&{[u]}={[v]}\,\,\, {\rm in} \,\,\, K_1(A)\tand {\rm dist}(\bar{u}, {\bar v})<\dt,
\eneq
Then there exists a unitary  $w\in U( A )$ such that
\beq\label{QT-3}
\|w^*uw-v\|<\ep.
\eneq
\end{thm}

\begin{cor}\label{QTcor} {\rm (cf. Cor 3.2 of \cite{Lnexp2})}
Let $A$ be a unital separable simple \CA\, with $TR(A)\le 1.$
Let $u\in U(A)$ be a unitary with ${\rm sp}(u)=\T.$
Then, for any $\ep>0,$ there exists $\dt>0$ and $\gamma>0$ and an integer $N\ge 1$ satisfying the following:
if $v\in U(A),$
\beq\label{cqt-1}
[u]=[v],\,\, {\rm dist}({\bar u}, {\bar v})<\gamma\andeqn
|\tau(u^k)-\tau(u^k)|<\dt\tforal \tau\in T(A),
\eneq
$k=\pm 1,\pm 2,...,\pm N,$ there exists a unitary $w\in U(A)$ such that
\beq\label{cqt-2}
\|u-w^*vw\|<\ep.
\eneq
\end{cor}





\begin{lem}\label{Ttr1}
Let $A$ be a unital separable simple \CA\, with $TR(A)\le 1$ and let $u\in CU(A)$ be a unitary.
 Then, for any $\ep>0,$ there exists
a self-adjoint  element $h\in A_{s.a.}$ with $\|h\|\le 1$ such that
$\tau(h)=0$ for all $\tau\in T(A)$ and
$$
\|u-\exp(i2\pi h)\|<\ep.
$$
Moreover, if ${\rm sp}(u)=\T,$ we may assume that ${\rm sp}(h)=[-1, 1].$
As a consequence,
$$
{\rm cel}_{CU}(A)\le 2\pi.
$$

\end{lem}

\begin{proof}
If ${\rm sp}(u)\not=\T,$ then $u=\exp (i g(u))$ where $g$  is
a continuous branch of logarithm with $\|g\|\le 2\pi.$
Thus we  may assume that ${\rm sp}(u)=\T.$
Without loss of generality, we may assume that
$$
u=u_1u_2\cdots u_k,
$$
where each $u_j$ is a commutator of $U(A).$  As in 6.9 of \cite{Lntr1}, $u_j\in U_0(A).$
It follows from 6.9 of \cite{Lntr1} again that  there are $h_j\in A_{s.a.}$ such that
\beq\label{Ttr1-1}
u=\prod_{j=1}^m \exp(i h_j)
\eneq
with $\sum_{j=1}^m\|h_j\|\le 8\pi+1.$

Let $\ep>0.$
Let $\gamma>0,$ $\dt>0$  and $N$ be as required by \ref{QTcor} for $\ep/4$ (in place of $\ep$).
We may assume that $\dt<\ep/4.$ Let  $d=\min\{\dt, \gamma, \ep/2\}.$
Since $TR(A)\le 1,$ there
exists a projection $p\in A$ and a \SCA\, $B\in A$ with $1_B=p$
such that $B\cong \oplus_{i=1}^m C(X_i, M_{r(i)}),$ where
$X_i=[0,1]$ or a point, and
\beq\label{ex3-2}
&&\|px-xp\|<{d\over{2^7(\pi+1)Nmk}} \tforal x\in
\{u, u_j: j\le 1\le k \}\\\label{ex3-2+1}
&&\|(1-p)u(1-p)-(1-p)\prod_{j=1}^m\exp(i(1-p)h_j(1-p))\|<
{d\over{2^7(\pi+1)Nmk}}, \\\label{ex3-2+2}
&& pup,\,pu_jp\in_{{d\over{2^7(\pi+1)Nmk}}} B\andeqn
\tau(1-p)<{d\over{16(\pi+1)Nmk}}\tforal \tau\in T(A).
\eneq

There exists a  unitary $v_1\in B$ such that
\beq\label{ex3-3}
\|pup-v_1\|<{\dt\over{2^6(\pi+1)Nm}}
\eneq
which is a product of $k$ commutators in $B.$
Put  $v_2=(1-p)\prod_{j=1}^k\exp(i(1-p)h_j(1-p)).$ Since  $v_1\in
B$ and $v_1$ is a product of commutators, in each summand of $B,$ determinant of $v_1$ at every point must be one.
It follows from  \ref{Det=1} that there exists a self-adjoint  element $b\in
B_{s.a}$ such that
\beq\label{ex3-4}
\|b\|\le 2\pi,\,\,\,t(b)=0\tforal  t\in T(B)\andeqn\\
\|u_1-p\exp(ib)\|<{d\over{2^7(\pi+1)Nm}}.
\eneq

We may assume that $1-p\not=0.$ Since $(1-p)A(1-p)$ is simple and
has (SP), we obtain two mutually orthogonal and mutually equivalent
projections $e_1, e_2\in (1-p)A(1-p).$ Suppose that $z\in U((1-p)A(1-p))$
such that $z^*e_1z=e_2.$ Let $b_0\in (e_1Ae_1)_{s.a.}$ with
${\rm sp}(b_0)=[-2\pi, 2\pi].$ Let $b_1=b_0-z^*b_0z$ and let
$b_2=b_1+b.$  Note that
\beq\label{Ttr1-10}
{\rm sp}(b_2)=[-2\pi, 2\pi]\andeqn \tau(b_2)=0\tforal \tau\in T(A).
\eneq
Define
$$
v_0=(1-p-e_1-e_2)+(p+e_1+e_2)\exp(i b_2).
\andeqn u_0=p\exp(ib)+v_2.
$$
Then, by
(\ref{ex3-2}), (\ref{ex3-2+1}), (\ref{ex3-3}) and (\ref{ex3-4}),
\beq\label{ex3-5-1}
\|u_0-u\|&<& \|(p\exp(ib)-pup)+(u_2-(1-p)u(1-p))\|\\
&&+\|(pup-(1-p)u(1-p))-u\|\\
\label{ex3-5+2}\label{Tr1-n}
&<&({d\over{2^6(\pi+1)N}}+{d\over{2^7(\pi+1)N}})+{2d\over{2^7N}}<
{5d\over{2^7(\pi+1)N}}.
\eneq
and
\beq\label{ex3-5}
u_0v_0^*=\prod_{j=1}^k \exp(i(1-p)h_j(1-p))\exp(-i b_1).
\eneq
Note that
\beq\label{ex3-6}
|\tau(\sum_{j=1}^k (1-p)h_j(1-p))|&\le &\sum_{j=1}^k |\tau((1-p)h_j(1-p))|\\
&=&\tau(1-p)(\sum_{j=1}^k\|h_j\|+2\pi)<5d/2^7N
\eneq
for all $\tau\in T(A).$
It follows that
\beq\label{ex3-7}
{\rm dist}({\bar u_0}, {\bar v_0})<d/4N\,\,\,{\rm in}\,\,\, U_0(A)/CU(A).
\eneq
It follows from (\ref{ex3-5+2}) and (\ref{ex3-7}) that
\beq\label{ex3-8}
{\rm dist}({\bar u}, {\bar v_0})<d/2N.
\eneq
On the other hand, for each $s=1,2,...,N,$ by  (\ref{ex3-5}), (\ref{ex3-5+2}) and (\ref{ex3-2+2})
\beq\label{ex3-9}
|\tau(u^s)-\tau(v_0^s)|&\le &
|\tau(u^s)-\tau(u_0^s)|+|\tau(u_0^s)-\tau(v_0^s)|\\\label{ex3-9+}
&\le & \|u^s-u_0^s\|+2\tau(1-p)\\
& \le & N\|u-u_0\|+2  \tau(1-p)\\
&<& {5d\over{2^7(\pi+1)}}+{d\over{(\pi+1)N}}<d
\eneq
for all $\tau\in T(A).$ From the above inequality and  (\ref{ex3-8})
and applying \ref{QTcor}, one obtains a unitary $W\in U(A)$ such
that
\beq\label{ex3-10}
\|u-W^*v_0W\|<\ep.
\eneq

Put $h=W^*b_2W.$  Then
\beq\label{ex3-11}
&&{\rm sp}(h)=[-2\pi, 2\pi],\tau(b)=0\tforal \tau\in T(A)\andeqn\\
&&\|u-\exp(ih)\|<\ep.
\eneq

\end{proof}

\begin{thm}\label{MT1}
Let $A\in {\cal A}_1$ be a unital separable simple ${\cal Z}$-stable \CA. Suppose that $u\in CU(A).$ Then, for any $\ep>0,$ there exists a self-adjoint element $h\in A$ with $\|h\|<1$ such that
\beq\label{MT1-1}
\|u-\exp(i2\pi h)\|<\ep
\eneq
In particular, ${\rm cel}_{CU}(A)\le 2\pi.$
\end{thm}

\begin{proof}
We may assume that ${\rm sp}(u)=\T.$  Let $\ep>0.$ Let $\phi: C(\T)\to A$ be defined by $\phi(f)=f(u).$ It is a unital monomorphism.
It follows from Proposition 11.1 of \cite{Lnappeqv} that there is a non-decreasing function
$\Delta: (0,1)\to (0,1)$ such that
\beq\label{MT1-1+}
\mu_{\tau}(O_a)\ge \Delta (a)\tforal \tau\in T(A)
\eneq
for all open balls $O_a$ of $\T$ with radius $a\in (0,1).$

Note, by \cite{LS}, for any supernatural number $\p$ of infinite type,
$TR(A\otimes M_\p)\le 1.$
Consider $u\otimes 1.$ Denote by $u_\p$ for
$u\otimes 1$ in $A\otimes M_\p.$
For any $\ep/2>\ep_0>0,$ by \ref{Ttr1}, there is a
self-adjoint  element $h_\p\in A\otimes M_\p$ with
${\rm sp}(h_\p)=[-2\pi, 2\pi]$ such that
\beq\label{MT1-2}
\|u_\p-\exp(ih_\p)\|<\ep_0 \andeqn \tau(h_\p)=0\tforal \tau\in T(A\otimes M_\p).
\eneq
Let $\psi_0: C(\T)\to A\otimes M_\p$ be the \hm\, defined by
$\psi_0(f)=f(\exp(ih_\p))$ for all $f\in C(\T).$
Let $\eta>0,$ $\dt>0$ and let ${\cal G}$ be a finite subset as required by \ref{Appu} for $\ep/2$ (in place of $\ep$) and $\Delta.$
Choose $\ep_0$ sufficiently small, so the following holds:
For any unitary $v\in A\otimes M_\p,$ if
$\|u_\p-v\|<\ep_0,$ then
\beq\label{MT1-4}
|\tau(g(u_\p))-\tau(g(v))|<\dt\tforal \tau\in T(A)
\eneq
and for all $g\in {\cal G}.$
Note each $\tau\in A\otimes M_\p$ may be written as
$s\otimes t,$ where $s\in T(A)$ is any tracial state and
$t\in T(M_\p)$ is the unique tracial state.

Let $\Gamma: C([-2\pi, 2\pi])_{s.a.}\to \aff(T(A)$ be defined
by
\beq\label{MT1-5}
\Gamma(f)(\tau)=(\tau\otimes t)(f(h_\p))\tforal f\in C([-2\pi, 2\pi])_{s.a.}
\eneq
and for all $\tau\in T(A),$
where
$t$ is the unique tracial state on $M_\p.$

It follows from \ref{embedding}  that there exists a self-adjoint
element $h\in A$  with ${\rm sp}(h)=[-2\pi, 2\pi]$ such that
\beq\label{MT1-6}
\tau(f(h))=\Gamma(f)(\tau)=(\tau\otimes t)(f(h_\p))\tforal f\in C([-2\pi, 2\pi])
\eneq
and for all $\tau\in T(A).$
In particular,
\beq\label{MT1-7}
\tau(h)=0\tforal \tau\in T(A).
\eneq
Define $v_1=\exp(i h)\in A.$
Note that, by (\ref{MT1-5}),
\beq\label{MT1-8}
\tau(g(v_1))=(\tau\otimes t)g(\exp(ih_\p))\tforal \tau\in T(A)
\eneq
and for all $f\in C(\T).$ By the choice of $\ep_0,$ as
in (\ref{MT1-4}),
\beq\label{MT1-10}
|\tau(g(u))-\tau(g(v_1))|
=|(\tau\otimes t)(g(u_\p))-(\tau\otimes t)(g(\exp(ih_\p)))|<\dt
\eneq
for all $\tau\in T(A)$ and for all $g\in {\cal G}.$
We also have  $[v_1]=[u]=0$ in $K_1(A).$ Furthermore, by (\ref{MT1-7}),
$v_1\in CU(A\otimes {\cal Z}).$ Thus, by applying (\ref{Appu}),
there exists a unitary $w\in A$ such that
\beq\label{MT1-11}
\|u-w^*\exp(i h) w\|<\ep/2
\eneq

\end{proof}

\begin{thm}\label{NNZlT}
Let $A$ be a unital separable simple ${\cal Z}$-stable \CA\, in ${\cal A}_0$ with a unique  tracial state.
Then, for any unitary $u\in U_0(A),$ there exists
 a real number $-\pi<a<\pi$ such that, for any $\ep>0,$ there exists a self-adjoint element $h\in A$  with  $\|h\|\le 2\pi$ and
$$
\|u-\exp(i(h+a))\|<\ep.
$$
 Consequently
$$
{\rm cel}(A)\le 3\pi.
$$
\end{thm}

\begin{proof}
Let $u\in U_0(A)$ and let $\ep>0.$
Since $A$ has a unique tracial state $\tau,$
$U_0(A)/CU(A)=\R/\overline{\rho_A(K_0(A))}$ and $\Z\in \rho_A(K_0(A)).$ 
Therefore there is $t\in (-1,1)$ such that
\beq\label{NNZ-n-1}
{\rm Det}(u)=t+\overline{\rho_A(K_0(A))}.
\eneq
Consequently
\beq\label{NNZ-n-2}
e^{-\pi t}u\in CU(A).
\eneq
It follows from \ref{MT1} that there is a self-adjoint element $h\in A$ with $\|h\|\le 2\pi$ such that
\beq\label{NNZ-n-3}
\|e^{-\pi t}u-\exp(ih)\|<\ep
\eneq
Therefore
\beq\label{NNZ-n+4}
\|u-e^{i\pi t} \exp(ih)\|<\ep.
\eneq
Let $a=\pi t.$
Note that $e^{i \pi t}\exp(ih)=\exp(i(h+a)).$  Put $h_1=h+a.$
We conclude that
$$
\|u-\exp(ih_1)\|<\ep.
$$
Note that $\|h_1\|\le \|h\|+|a|<3\pi.$
There is $h_2\in A_{s.a.}$ with $\|h_2\|<2\arcsin (\pi/2)$ such
that
\beq\label{NNZ-n+5}
u=\exp(ih_1)\exp(ih_2).
\eneq
If we choose $\ep$  so that
$$
2\arcsin(\ep/2)<3\pi-\|h\|+|a|,
$$
Then
$$
\|h_1\|+\|h_2\|<3\pi.
$$


\end{proof}

\begin{cor}\label{JSl}
Let $u\in U_0({\cal Z})$ be a unitary. There exists $t\in (-\pi,\pi )$ such that, for any  $\ep>0,$  there exists
a self-adjoint element $h\in {\cal Z}$ with $\|h\|\le 2\pi$ satisfying
\beq\label{JS-1}
\|e^{it}u-\exp(ih))\|<\ep.
\eneq
\end{cor}

\section{Examples}

\begin{NN}\label{uniexam}
{\rm Let $u\in C([0,1], M_n)$ be  defined as follows:
\beq\label{C01l-1}
u(t)=e^{\pi i t(2-1/(n-1))}e_1+e^{-\pi i{t(2-1/(n-1))\over{(n-1)}}}(\sum_{k=2}^ne_k)\andeqn\\
h(t)=t(2-1/(n-1))e_1-{t(2-1/(n-1))\over{(n-1)}}(\sum_{k=2}^ne_k)
\tforal t\in [0,1],
\eneq
where $\{e_1, e_2,...,e_n\}$ is a  set of mutually orthogonal rank one projections.

Then
$$
u(t)=\exp(i\pi h) \andeqn \tau(h)=0\tforal \tau\in T(C([0,1],M_n)).
$$
Therefore ${\rm det}(u(t))=1$ for all $t\in [0,1]$ and $u\in CU(C[0,1], M_n).$ Note also $\|h\|=\pi(2-1/(n-1)).$
In what follows we  will show that ${\rm cel}(u)\ge (2-1/(n-1))\pi.$
It should be noted that it is much easier to show that if $u(t)=\exp(iH)$ for some
self-adjoint element in $C([0,1], M_n)$ then $\|H\|\ge (2-1/(n-1))\pi.$

Suppose that ${\rm cel}(u)=r_1.$ Fix $r_1>\ep>0$  and put $r=r_1+\ep/16.$
Then there are self-adjoint  elements $h_1, h_2,...,h_k\in C([0,1], M_n)$ such that
\beq\label{C01l-3}
u=\prod_{j=1}^k \exp(i h_j)\andeqn \sum_{j=1}^k \|h_j\|=r.
\eneq
Define
$u_s=\prod_{j=1}^k \exp(i h_j(1-s)).$ Then $u_s$ is continuous and piecewise smooth on $[0,1].$
Moreover ${\rm length}\{u(t)\}\le r.$
Since $h_j(t)(1-s)$ is continuous on $[0,1]\times [0,1],$
one shows that
 $W(t,s)=u_s(t)$ is continuous on $[0,1]\times [0,1].$

Furthermore
\beq\label{C01l-4}
\|u_{s_1}-u_{s_2}\|\le r|s_1-s_2|\tforal s_1, s_2\in [0,1].
\eneq


}

\end{NN}

\begin{lem}\label{specdist}
Let $u$ and $v$ be two unitaries in a unital \CA\, $A.$ Suppose that there is a continuous path
of unitaries $\{w(t): t\in [0,1]\}\subset A$ with $w(0)=u$ and $w(1)=v.$
Then,  if $\lambda\in {\rm sp}(u),$ there is a continuous path $\{\lambda(t)\in \T: t\in [0,1]\}$
such that
$\lambda(0)=\lambda,$ $\lambda(t)\in {\rm sp}(w(t))$ for all $t\in [0,1].$

If furthermore, ${\rm length}\{w(t): t\in [0,1]\}=r\le \pi/2,$
then one can require that
$$
{\rm length}\{\lambda(t): t\in [0,1]\}\le r.
$$
\end{lem}

\begin{proof}
The proof of this was originally taken from an argument of Phillips.
As in Lemma 4.2.3 of \cite{Lnbk}, one obtains a sequence of
 partitions $\{{\cal P}_n\}$ of $[0,1]$ such that ${\cal P}_n\subset {\cal P}_{n+1},$ $n=1,2,...,$
 for each partition ${\cal P}_n=\{0=t_0^{(n)}<t_1^{(n)}<\cdots t_{k(n)}^{(n)}=1\},$
 there are $\lambda(n,i)\in {\rm sp}(w(t_i^{(n)}))$ such that
 \beq\label{specdist-1}
 |\lambda(n,i)-\lambda(n, i+1)|= \|w(t_i^{(n)})-w(t_{i+1}^{(n)})\|\andeqn\\
\sum_{i=1}^{k(n)}|\lambda(n,i)-\lambda(n, i+1)|\le |\sum_{i=1}^{k(n)}\|w(t_i^{(n)})-w(t_{i+1}^{(n)})\|\le r,
 \eneq
 if $\{w(t)\}$ is rectifiable with
 ${\rm length}\{w(t): t\in [0,1]\}=r.$
 Write $\lambda(n,j)=e^{i \theta(n,j)}$ with $\theta(n,j)\in [0, 2\pi),$ $j=1,2,...,k(n)$ and $n=1,2,....$
Define
\beq\label{specdist-2}
\theta(t)=\sup\{ \theta(n,j): t_j^{(n)}\le t\}.
\eneq
By the uniform continuity of $w(t),$ one checks that $\lambda(t)=\exp(i\theta(t))$ is continuous on $[0,1],$ $\lambda(t)\in {\rm sp}(w(t))$ and
${\rm length}(\{\lambda(t)\}\le r.$
\end{proof}

\begin{NN}\label{Genuni}
{\rm
Suppose that
$u(t)\in C([c,d], M_n)$ is  a unitary which has the form:
$$
u(t)=f(t)q_1+z(t)\tforal t\in [c, d],
$$
where $f(t)\in C([0,1], \T),$ $q_1$ is a rank one projection
and $z\in (1-q_1)C([c,d], M_n)(1-q_1)$ is a unitary.
Let ${\rm cel}(u)=r_1$ and  fix $r_1/2>\ep>0.$ Let $r=r_1+\ep/16.$
let $\{W(t,s): s\in [0,1]\}$ be a continuous rectifiable path  defined
in \ref{uniexam} such that $W(t,s)\in C([c,d]\times [0,1], M_n)$ with $W(t,0)=u(t)$ and $W(t,1)=1$ with length $r.$

Fix $s_0\in (0,1].$ Suppose that ${\rm length}\{W(t,s): s\in [0,s_0]\}=r_0.$ Define $S_1$ the subset of $\T$ such that every point
of $S_1$ can be connected to a point in ${\rm sp}(z)$ by a continuous path of length at most $r_0.$

Then we have the following:
}
\end{NN}

\begin{lem}\label{seper}
Let $\{f(t): t\in [c, d]\}=\{e^{it} : t\in [t_0, t_1]\}$ with $f(c)=e^{it_0}$ and $f(d)=e^{it_1}$ such that
$t_1-t_0=r_2.$
Suppose that
$$
{\rm dist}(\lambda_t(s), S_1)>0\tforal t\in [t_0, t_1]\andeqn s\in [0, s_0],
$$
where ${\rm length}\{W(t,s): s\in [0, s_0]\}=r_0<r_2/2.$
Then
$$
W(t, s_0)=g_{s_0}(t)q_1+v_1(t)\tforal t\in [c_1, d_1]
$$
for some $[c_1, d_1]\subset [c, d]$ with $d_1>c_1.$
Moreover $g_{s_0}(t)$ is a continuous function  and
$$
\{g_{s_0}(t): t\in [c_1, d_1]\}\supset \{e^{it}: t\in [t_0+r_0, t_1-r_0]\}.
$$
\end{lem}

\begin{proof}
View $Z=W(t, s)|_{[c, d]\times [0, s_0]}$ as a unitary in $C([c, d]\times [0, s_0], M_n).$
Then the  assumption implies that
$$
{\rm sp}(Z)\subset  J\sqcup S_1,
$$
where $J=\{\lambda_t(s): t\in [t_0, t_1]\andeqn s\in [0, s_0]\}, $  Note that $J\cap S_1=\emptyset.$
Then, there is a non-zero projection $q_1'\in C([c, d]\times [0, s_0], M_n)$ such that
\beq\label{sep-1}
Z=z_1+z_2,
\eneq
where $z_1\in q_1'C([c, d]\times [0, s_0], M_n)q_1'$ and $z_2\in (1-q_1')C([c, d]\times [0, s_0], M_n)(1-q_1')$ are
unitaries such that ${\rm sp}(z_1)\subset J$ and ${\rm sp}(z_2)\subset S_1.$
Since $q_1'$ has rank one in $[c, d]\times \{0\},$ we conclude that $q_1'$ has rank one everywhere.
Thus
$$
Z(t,s)=g_s(t)q_1'(t,s)+z_2(s,t)\tforal (t,s)\in [c,d]\times [0, s_0].
$$
Note that $g_s(t)\in C([c,d]\times [0,s_0]).$
Therefore
$\{g_{s_0}(t): t\in [c,d]\}$ is an arc containing $g(c)=\lambda_c(s_0)$ and $g_{s_0}(d).$
By the assumption and \ref{specdist}, $g_{s_0}(c)\in \{e^{it}: t\in [t_0-r_0, t_0+r_0]\}$ and
$g_{s_0}(d)\in \{e^{it}t\in [t_1-r_0, t_1+r_0]\}.$
The lemma follows.
\end{proof}

\begin{lem}\label{connt}
Suppose that ${\rm length}(\{W(t, s): s\in [0, s_1]\}=C_1<\pi/4.$
If $[c, d]\subset [a, b]$ such that
$$
{\rm dist}(\{f(t): t\in [c, d]\},\{{\rm sp}(v_1(t)): t\in [c, d]\})=r_1=4\sin(C_1/2)+\dt
$$
for some $0<\dt<\pi/8,$ $\{f(t): t\in [c, d]\}=\{e^{it}: t\in [t_0, t_1]\}$ with $t_1-t_0>2r_1,$
then, for any $\dt>0,$  there exists an interval $[c_1, d_1]\subset  [c, d]$ with  $c_1<d_1,$ a  rank one projection $q_1\in C([c_1, d_1], M_n)$
such that
\beq\label{connt-1}
W(t, s_1)=g_{s_1}(t)q_1+v_1'(t),
\eneq
where $g_{s_1}(t)\in C([c_1, d_1])$ with
\beq\label{connt-2}
\{g_1(t): t\in [c_1, d_1]\}=\{e^{it}: t\in [t_0+C_1+\dt, t_1-C_1-\dt]\},
\eneq
where $2\pi>t_1'>t_0'\ge 0,$  and where $v_1'\in (1-q_1)C([c_1, d_1], M_n)(1-q_1)$ is a unitary
with ${\rm sp}(v_1(t))\subset S_1,$ where $S_1$ is a subset of $\T$ such that every point of $S_1$ can be connected
by a point in ${\rm sp}(v(t))$ ($t\in [c, d]$) by a continuous path with length at most $C_1.$
\end{lem}

\begin{proof}

Let $S_1$ be the subset of $\T$ such that every point in $S_1$ can be connected to a point
in $sp(v_1)$ with length at most $C_1.$
Since ${\rm length}(\{W(t, s): s\in [0, s_1]\}=C_1<\pi/4,$
\beq\label{connt-3}
{\rm dist}(\lambda_t(s), S_1)>0\tforal t\in [t_0, t_1]\andeqn s\in [0, s_1].
\eneq
Thus this lemma follows from \ref{seper} by taking smaller interval.
\end{proof}

\begin{thm}\label{C01L1}
Let $u(t)\in C([0,1], M_n)$ be the unitary defined in \ref{uniexam}.
For any $\ep>0,$
\beq\label{nc01T-1}
{\rm length}\{u(t)\}\ge \pi(2-{1\over{n-1}})-\ep.
\eneq
If $h\in C([0,1], M_n)_{s.a.}$ such that
\beq\label{nc01T-2}
\|u-\exp(ih)\|<\ep,
\eneq
then
$\|h\|\ge \pi(2-{1\over{n-1}})-2\arcsin(\ep/2).$
Moreover,
$$
{\rm cel}_{CU}(C([0,1], M_n))\ge \pi(2-{1\over{n-1}}).
$$
\end{thm}

\begin{proof}
Let $u$ be as in \ref{uniexam}.  Fix $\ep>0.$  We will keep notation above.
We write
$$
u(t)=f(t)e_1+v(t).
$$

Let $0<d\le \ep/2$ and $k\ge 1$ be an integer such that $kd=\pi(1-1/n-1).$
Let $0<a_0<b_0<1$ such that
$$
f(a_0)=e^{id+\ep/k2}\andeqn f(b_0)=e^{i\pi(2-1/n-1))-d-\ep/k2}.
$$
Let $0<s_0<1$ such that ${\rm length}\{W(t, s): s\in [0, s_0]\}=d.$
It follows from \ref{connt} that there exists $a_0<a_1<b_1<b_0$ such that
\beq\label{nc01L-3}
W(t,s_0)=g_1(t)q_1+z_1(t)\tforal  t\in [a_1, b_1],
\eneq
where $q_1$ is a rank one projection, $z_1\in (1-q_1)C([a_1, b_1], M_n)(1-q_1),$
$$
\{g_1(t): t\in [a_1, b_1]\}=\{e^{it}: t\in [2d+\ep/k2, \pi(2-1/n-1))-2d-\ep/k2)]\}
$$
with $g_1(a_1)=e^{i(2d+\ep/k2)}$ and $g_1(b_1)=e^{i (\pi(2-1/n-1))-2d-\ep/k2)},$
${\rm sp}(z_1)\subset S_1,$
where $S_1$ is the subset of $\T$ such that every point in $S_1$ is connected by a rectifiable
continuous path from $\{e^{it}: t\in [-{(2-1/n-1)\pi\over{n-1}}, 0]\}.$
In particular,
\beq\label{nc01L-4}
S_1\subset \{e^{it}: t\in [-{(2-1/(n-1))\pi\over{n-1}}-d, +d].
\eneq
Let $1>s_1>s_0$ such that ${\rm length}\{W(t, s): s\in [s_0, s_1]\}=d.$  By repeating above,
one obtains $a_1<a_2<b_2<b_1$ such that
\beq\label{nc01L-5}
W(t,s_1)=g_2(t)q_2+z_2(t)\,\,\, t\in [a_2, b_2],
\eneq
where $q_2$ is a rank one projection, $z_2\in (1-q_2)C([a_2, b_2], M_n)(1-q_2),$
\beq\label{nc01L-6}
&&\hspace{-0.8in}\{g_2(t): t\in [a_2, b_2]\}=\\
&&\hspace{-0.3in}\{e^{it}: t\in [2d+\ep/k2+d+\ep/k4, \pi(2-1/(n-1))-2d-\ep/k2-d-\ep/k4)]\}\\
&&=\{e^{it}: t\in [3d+\ep/k2+\ep/k4, \pi(2-1/(n-1))-3d-\ep/k2-\ep/k4)]\}
\eneq
with $g_2(a_2)=e^{i(3d+\ep/k2+\ep/k4})$ and $g_2(b_2)=e^{i (\pi(2-1/n-1))-3d-\ep/k2-\ep/k4)},$
${\rm sp}(z_1)\subset S_2,$
where $S_2$ is the subset of $\T$ such that every point in $S_2$ is connected by a rectifiable
continuous path from $\{e^{it}: t\in [-2{(1-1/n-1)\pi\over{n-1}}-d, d]\}.$
In particular,
\beq\label{nc01L-7}
S_2\subset \{e^{it}: t\in [-{(2-1/(n-1))\pi\over{n-1}}-2d, 2d]\}.
\eneq
By repeating this argument $k-1$ times,
We obtain $1>s_{k-1}>s_{k-2}$ such that
\beq\nonumber
&&\hspace{-0.3in}{\rm length}\{W(t, s): s\in [0, s_{k-1}]\}=(k-1)d=\pi(1-1/(n-1))-d\andeqn\\\nonumber
&&\hspace{-0.3in}\pi\in \{e^{it}: t\in [(k-1)d+\sum_{j=1}^{k-1}\ep/k2^j,  (2-1/(n-1))\pi-(k-1)d+\sum_{j=1}^{k-1}\ep/k2^j)]\}\subset {\rm sp}(W(t,s_{k-1})).
\eneq
Thus the minimum length of continuous path from $W(t, s_{k-1})$ to $1$ is at least $\pi.$
Thus
\beq\label{nc01L-9}
{\rm length}\{u(t)\} +\ep/16 &\ge & \pi+(k-1)d\\
&= &\pi +\pi(1-1/(n-1))-d\ge \pi(2-1/(n-1))-\ep/2
\eneq
for all $\ep>0.$
It follows that
\beq\label{nc01L-10}
{\rm length}\{u(t)\}\ge \pi(2-1/(n-1)).
\eneq

 \end{proof}

\begin{NN}\label{Ex2}

{\rm
Fix an integer $n>12$ and let $k_0\ge 0.$
Suppose that $0\le k\le k_0.$
Let $N=mn +k$  and let $N_0=mn.$
Consider a unitary $u_{00}\in C([0,1], M_{N_0}):$
\beq\label{Ex2-1}
u_{00}=e^{\pi i t(2-1/(n-1))}P_1+e^{-\pi i{t(2-1/(n-1))\over{(n-1)}}}P_2,
\eneq
where $P_1, P_2\in C([0,1], M_{N_0})$ are constant projections with ${\rm rank}P_1=m$ and ${\rm rank}P_2=(n-1)m.$
Define
\beq\label{Ex2-2}
u_0=u_{00}+v_0\in C([0,1], M_N),
\eneq
where $v_0\in (1-(P_1+P_2))C([0,1], M_N)(1-(P_1+P_2))$ is another unitary
such that ${\rm det}(v_0(t))=1$ for each $t\in [0,1]$ and
$v_0=\sum_{j=1}^{k} \lambda_j e_j,$ where
$\{e_1, e_2,...,e_k\}$ is a set of mutually orthogonal rank one constant projections in $(1-P_1-P_2)C([0,1], M_N)(1-P_1-P_2).$  Note
that ${\rm rank}(1-(P_1+P_2))=k.$

}
\end{NN}

\begin{lem}\label{closespec}
Let $n\ge 1$ be a given integer and let $\ep>0.$
There exists $\dt>0$ satisfying the following:
Choose $m_0>128(k_0+1)n\pi /\ep,$
for any unitary $u$ given in \ref{Ex2} with $m\ge m_0$ and any choice of $v_0$ as in \ref{Ex2}, if $v\in C([0,1], M_N)$ is another unitary
such that
$$
\|u-v\|<\dt,
$$
then
\beq\label{close-1}
&&|\mu_{tr, t, v}(\{e^{is}: s\in [t\theta_0-\ep/2, t\theta_0+\ep/2]\})-1/n|<\ep\andeqn\\
&&\hspace{-0.7in}|\mu_{tr, t, v}(\{e^{is}: s\in [-t\theta_0/(n-1)-\ep/2, -t\theta_0/(n-1)+\ep/2\})-(n-1)/n|<\ep
\eneq
for all $t\in [1/(n-1), 1],$
where $\mu_{tr, t, v}$ is the probability measure given by
$tr\circ \pi_t\circ \psi,$ where $\psi: C(\T)\to C([0,1], M_N)$ is the \hm\, defined by $\psi(f)=f(v)$ for all $f\in C(\T),$ and where
$tr$ is the normalized trace on $M_N.$
\end{lem}

\begin{proof}
Let $\ep>0.$ Choose $m_0\ge 1$ such that $\pi/m_0<\ep/128(k_0+1)n.$
Let $f_{1,t}, f_{2,t}\in C(\T)$ be defined
by $f_{j,t}(s)=1,$ if $s\in \{e^{i\theta}: \theta\in [t-\ep/2^{j}, t+\ep/2^{j}]\},$ $f_{j,t}(s)=0$ if $s\not\in
\{e^{i\theta}: \theta\in (t-\ep/2^{j-1}, t+\ep/2^{j-1})\}$ and linear in between,
$j=1,2.$ Choose $\dt_1>0$ satisfying the following:
if $v', v''$ are two unitaries in any unital \CA\, with
$\|v'-v''\|<\dt_1$ then
\beq\label{close-2}
\|f_{j, e^{i\pi/2}}(v')-f_{j, e^{i\pi/2}}(v'')\|<\ep/32.
\eneq
Note that, for any $\theta\in [0, 2\pi),$
\beq\label{close-3}
f_{j, e^{i\pi/2}}(se^{-i\theta})=f_{j,e^{i(\theta+\pi/2)}})(s),\,\,\,j=1,2.
\eneq
Therefore
\beq\label{close-2}
\|f_{j, e^{i\theta}}(v')-f_{j, e^{i\theta}}(v'')\|
=\|f_{j, e^{i\pi/2}}(v'e^{i(\theta-\pi/2)})-f_{j, e^{i\pi/2}}(v''e^{i(\theta-\pi/2)})\|<\ep/32,
\eneq
since
$$
\|v' e^{i(\theta-\pi/2)}-v''e^{i(\theta-\pi/2)}\|=\|v'-v''\|<\dt_1.
$$

Therefore, if $\|u-v\|<\dt,$
\beq\label{close-3}
|\tau(f_{j,t}(u))-\tau(f_{j,t}(v))|<\ep/32
\eneq
for all $t\in \T$ and for all $\tau\in T(C([0,1], M_N),$ $j=1,2.$
Thus, in particular,
\beq\label{close-4}
|tr\circ \pi_t(f_{j, s}(u))-tr\circ \pi_t(f_{j,s}(v))|<\ep/16
\eneq
for all $s\in \T, t\in [0,1],$
$j=1,2.$
Note that
\beq\label{close-4+}
|{m\over{N}}-1/n|&=&|{mn\over{(mn+k_0)n}}-
{mn+k_0\over{(mn+k_0)n}}|\\
&=&{k_0\over{Nn}}.
\eneq
It follows that
\beq\label{close-5}
\hspace{-0.4in}|tr\circ \pi_t (f_{j,e^{it(2-1/(n-1))\pi}}(u))-1/n|<k_0/N+k_0/Nn<k_0/mn+k_0/Nn<\ep/64
\eneq
for all $t\in [1/(n-1), 1],$ $j=1,2.$
Note also that
\beq\label{close-5+}
tr\circ \pi_t(f_{2,e^{it(2-1/(n-1))\pi}}(u))\le
\mu_{tr, v,t}(I_t)\le  tr\circ \pi_t(f_{1,e^{it(2-1/(n-1))\pi}}(u)),
\eneq
where  $I_t=\{e^{is}: s\in [t\theta_0-\ep/2, t\theta_0-\ep/2]\}.$
Combing this with (\ref{close-5}) and (\ref{close-4}), we obtain that
\beq\label{close-6}
|\mu_{tr, v,t}(I_t)-1/n|<5\ep/64
\eneq
for all $t\in [1/(n-1), 1].$  Similarly,
\beq\label{close-7}
&&|tr\circ \pi_t(f_{j, e^{it(2-1/(n-1))\pi/(n-1)}})-(n-1)/n|<k_0/nm+k_0/Nn\\
&&\andeqn
|\mu_{tr, v, t}(J_t)-(n-1)/n|<5\ep/54,
\eneq
where
$$
J_t=\{e^{is}: s\in [-t\theta_0/(n-1)-\ep/2, -t\theta_0/(n-1)+\ep/2]\}.
$$
\end{proof}

\begin{lem}\label{Ex2leng}
Let $n\ge 12.$ There exists $\dt>0$ and integer $m_0>2^{15}(k_0+1)n^3\pi^2$ satisfying the following:
If $u$ is as in \ref{Ex2} and $h\in C([0,1], M_N)_{s.a.}$ with
$\|h\|\le 2\pi$ such that
\beq\label{Ex2len-1}
\|u-\exp(ih)\|<\dt,
\eneq
then
\beq\label{Ex2len-2}
\|h\|\ge 2(1-1/(n-1))\pi.
\eneq
\end{lem}

\begin{proof}
Let $\theta_0=(2-1/(n-1))\pi.$
Let $\ep=1/2^8\pi n^2.$ Choose $\dt>0$ (in place of $\dt$) and $m_0\ge 128(k_0+1)n\pi/\ep=2^{15}(k_0+1)n^3\pi^2$ be as required by
\ref{closespec} for the given $\ep$ and $n.$ We also require
that $\dt<\ep/64.$
Note that ${\rm det}(u(t))=1$ for all $t\in [0,1].$

Suppose that $h\in C([0,1], M_N)_{s.a.}$ with
\beq\label{Ex2len-3}
\|u-\exp(i h)\|<\dt\andeqn \|h\|<2(1-1/(n-1))\pi.
\eneq
Then
\beq\label{Ex2len-3+}
u=\exp(ih_0)\exp(ih)
\eneq
for some self-adjoint  $h_0\in C([0,1], M_N)$ with
\beq\label{Ex2len-3+1}
\|h_0\|<2\arcsin (\ep/128).
\eneq
It follows that
$$
tr(h_0(t))+tr(h(t))=2L\pi /N
$$
for every $t\in [0,1]$  and for some integer $L.$
Therefore
\beq\label{Ex2len-3+2}
|tr(h(t))-2L\pi/N|<2\arcsin(\ep/128)
\eneq
for all $t\in [0,1].$

Let  $\{a_1,a_2,...,a_N\}$ be the  set of eigenvalues of $h(1)$
counting multiplicity.

Let $A_1=\{e^{i a_1}, e^{ ia_2},...,e^{i a_N}\}\cap I_1,$ $A_2=\{e^{ia_1}, e^{i a_2},...,e^{i a_N}\}\cap J_1$ and $A_3=\{e^{ia_1}, e^{ia_2},...,e^{ia_N}\}\cap (\T\setminus (I_1\cup J_1)),$ where
\beq\label{Ex2len-4}
I_1&=&\{e^{is}: s\in [\theta_0-\ep/2, \theta_0+\ep/2]\}\andeqn\\
J_1&=&\{e^{is}: s\in [-\theta_0/(n-1)-\ep/2,  -\theta_0/(n-1)+\ep/2]\}.
\eneq
Since $\|h\|<2(1-1/(n-1))\pi,$
\beq\label{Ex2len-5}
a_j&<&0,\,\,\,{\rm if}\,\,\, e^{ia_j}\in A_1\andeqn\\
a_j&<&0,\,\,\,{\rm if}\,\,\, e^{ia_j}\in A_2.
\eneq

Put $X_N=\{1,2,...,N\}.$
Define $H_1\in L^1(X_N,\mu)$ with normalized counting measure $\mu$ on $X_N$
by $H_1(j)=a_j,$ $j=1,2,...,N,$ and
\beq\label{Ex2len-6}
F_1(j)&=&{-\pi\over{n-1}},\,\,\, {\rm if}\,\,\, e^{ia_j}\in A_1,\\
F_1(j)&=&{-\theta_0\over{n-1}},\,\,\, {\rm if},\,\,\, e^{ia_j}\in A_2
\eneq
and $F_1(j)=0$  if $e^{ia_j}\in A_3.$
By \ref{closespec},
\beq\label{Ex2len-7}
\int_N |e^{i H_1}-e^{iF_1}|d\mu&=&
\int_{e^{ia_j}\in A_1} |e^{ia_j}-e^{-i\pi/(n-1)}|d\mu(j)\\
&&+\int_{e^{ia_j}\in A_2} |e^{ia_j}-e^{-i \theta_0/(n-1)}|d\mu(j)\\
&&+\int_{e^{ia_j}\in A_3}|e^{ia_j}-1|d\mu(j)\\\label{Ex2len-7+}
&<&\ep/2+\ep/2+2k_0/N<2\ep.
\eneq
We also have
\beq\label{Ex2len-8-}
(1/n)\sum_{j=1}^N F_1(j)={-\pi\over{n-1}}\mu(A_1)+{-\theta_0\over{n-1}}\mu(A_2).
\eneq
By \ref{closespec}, we estimate that
\beq\label{Ex2len-8}
|(1/n)\sum_{j=1}^N F_1(j)-({-\pi/(n-1))\over{n}}+{-\theta_0\over{n}})|\\
<|{-\pi/(n-1))\over{n}}|\ep+|{-\theta_0\over{n}}|\ep<2\ep,
\eneq
or,
\beq\label{Ex2len-9}
|(1/n)\sum_{j=1}^n F_1(j)+2\pi/n|<2\ep.
\eneq
It follows from (\ref{Ex2len-3+2}), (\ref{Ex2len-9}) and (\ref{Ex2len-7})-
(\ref{Ex2len-7+}) that
\beq\label{Ex2len-10}
|L/N+1/n|<4\ep+\ep/64<1/64n^2
\eneq
Let $\{b_1, b_2,...,b_N\}$ be the set of eigenvalues of $h(1/n).$
Let $B_1=\{b_1, b_2,...,b_N\}\cap I_{1/n},$
$B_2=\{b_1, b_2,...,b_N\}\cap J_{1/n}$ and
$B_3=\{b_1,b_2,...,b_N\}\cap (\T\setminus I_{1/n}\cup J_{1/n}),$
where
\beq
I_{1/n}=\{e^{it}: t\in [\theta_0/n-\ep/2, \theta_0/n+\ep/2]\}\andeqn\\
J_{1/n}=\{e^{it}: t\in [-\theta_0/n(n-1)-\ep/2, -\theta_0/n(n-1)+\ep/2]\}
\eneq
Then, since $$\|h\|<2(1-1/(n-1))\pi,$$
\beq\label{Ex2len-11}
b_j&>&0,\,\,\,{\rm if}\,\,\, b_j\in B_1\andeqn\\
b_j&<&0\,\,\,\, {\rm if}\,\,\, b_j\in B_2.
\eneq
Define
\beq\label{Ex2len-12}
F_2(j)&=&(1/n)\theta_0 ,\,\,\,{\rm if}\,\,\, b_j\in B_1\\
F_2(j)&=& -(1/n)(\theta_0/(n-1)),\,\,\,{\rm if}\,\,\, b_j\in B_2\andeqn\\
F_2(j)=0
\eneq
By \ref{closespec},
\beq\label{Ex2len-13}
\int_N |e^{i b_j}-e^{iF_2(j)}|d\mu(j)&=&
\int_{e^{ib_j}\in B_1} |e^{ib_j}-e^{i\theta_0/n}|d\mu(j)\\
&&+\int_{e^{ib_j}\in B_2} |e^{ib_j}-e^{-i \theta_0/n(n-1)}|d\mu(j)\\
&&+\int_{e^{ib_j}\in B_3}|e^{ib_j}-1|d\mu(j)\\\label{Ex2len-13+}
&<&\ep/2+\ep/2+2k_0/N<2\ep.
\eneq
We have
\beq\label{Ex2len-14}
(1/n)\sum_{j=1}^N F_2(j)={\theta_0\over{n}}\mu(B_1)-{\theta_0\over{n(n-1)}}\mu(B_2).
\eneq
By \ref{closespec}, as above, we estimate that
\beq\label{Ex2len-14}
|(1/n)\sum_{j=1}^n F_2(j)-({\theta_0\over{n^2}}+{-\theta_0\over{n^2}})|<2\ep,
\eneq
or,
\beq\label{Ex2len-15}
|(1/n)\sum_{j=1}^n F_2(j)|<2\ep.
\eneq
It follows from (\ref{Ex2len-3+2}), (\ref{Ex2len-15}) and (\ref{Ex2len-13})-(\ref{Ex2len-13+}) that
\beq\label{Ex2len-16}
|L/N|<4\ep+\ep/64<1/64n^2.
\eneq
One then obtains a contradiction from (\ref{Ex2len-16}) and (\ref{Ex2len-10}).
Thus
\beq\label{Ex2len-17}
\|h\|\ge 2(1-1/(n-1))\pi.
\eneq
\end{proof}

\begin{lem}\label{Ex2ml}
Let $(G, G_+)$ be a countable  unperforated ordered  group.
Then there exists a unital simple \CA\, $A$ which is an inductive limit of interval algebras satisfying the following:

For any $\ep>0,$ there exists a unitary $u\in CU(A)$ and $\dt>0$ satisfying the following:
if $h\in A_{s.a.}$ with $\|h\|\le 2\pi$ such that
$$
\|u-\exp(ih)\|<\dt,
$$
then
$$
\|h\|\ge 2\pi-\ep.
$$
\end{lem}

\begin{proof}
Fix $1/2>\ep.$
Choose $n\ge 12$ such that $\pi/(n-1)<\ep/4.$
Let $k_0=n-1.$
Let $m_0'>2^{15}(k_0+1)n^3\pi^2$ (in place of $m_0$) be an integer required by \ref{Ex2leng} for the above mentioned $n$ and $k_0.$ Let $m_0=2m_0'.$

Let $C=\lim_{k\to\infty} (C_k,\phi_k)$ be a unital simple AF-algebra, where each $C_k$ is a unital finite dimensional \CA, such that
$(K_0(C), K_0(C)_+)=(G, G_+).$ We may assume that the map
$\phi_k: C_k\to C_{k+1}$ is   unital and injective.
We write
$$
C_k=M_{r(1,k)}\oplus M_{r(2,k)}\oplus \cdots M_{r(m(k),k)}.
$$
Let $\phi_{k,j}: M_{r(j,k)}\to C_{k+1}$ be the \hm\,
defined by $\phi_{k,j}=(\phi_k)|_{M_{r(j,k)}}.$
Define $\pi_{k,j}: C_k\to M_{r(j,k)}$ by the projection to the summand.
Set $\phi_{j,i,k}=\pi_{k+1, i}\circ \phi_{k,j}: M_{r(j,k)}\to M_{r(k+1, i)}.$
Note that $(\psi_{j,i,k})_{*0}$  is determined by its
multiplicity $M(j,i,k).$ Since $C$ is simple, without loss of generality, we may assume that
$r(j,1):=r(j)\ge 2n(m_0+1),$ $j=1,2,...,m(1).$ By passing to a subsequence if necessary, we may assume that
$M(i,j,k)\ge (2m_0+1)k.$
There is a set of $M(j,i,k)$ mutually orthogonal  projections  $\{e_{j,i,k,s}: s\}$ in $M_{r(i,k+1)}$ such that
each $e_{j,i,k}$ has rank $r(j,k),$
$$
\sum_{j=1}^{m(j)}\sum_{s=1}^{M(j,i,k)} e_{j,i,k,s}=1_{M_{r(i,k+1)}}.
$$
Put
$$
e(j,i,k)=\sum_{s=1}^{M(j,i,k)}e_{j,i,k,s}.
$$

We write
$$
r(j)=d(j)n+k(j),\,\,\, k(j)<n,
$$
where $d(j)\ge 2m_0.$

Denote $\theta_0=(2-1/(n-1))\pi,$ $j=1,2,...,m(1).$
Let $B_{j,k}=C([0,1], M_{r(j,k)}),$ $j=1,2,...,m(k),$ $k=1,2,....$
Let $\{t(0,k), t(1,k),...,t(k,k)\}$ be a partition
of $[0,1]$ such that $t(0,k)=0,$ $t(k,k)=1$ and
$t(i,k)-t(i-1,k)=1/(k+1),$ $i=1,2,...,k,$ $k=1,2,....$
Define $\psi_{j,i,k}: B_{j,k}\to e(j,i,k)B_{i,k+1}e(j,i,k)$ as follows:
\beq\label{2ml-1}
\psi_{j,i,k}(f)={\rm diag}(\underbrace{f,f,...,f}_{M(j,i,k)-k},f(t(1,k)),
f(t(2,k)),...,f(t(k,k)))
\eneq
for all $f\in M_{j,k}.$  Define $A_k=C([0,1])\otimes C_k.$
Note that
$$
A_k=\bigoplus_{j=1}^{m(k)}C([0,1], M_{r(j,k)}).
$$
Let $\psi_k: A_k\to A_{k+1}$ be the unital \hm\, given by the partial maps
$\psi_{j,i,k}.$
Define $A=\lim_{n\to\infty}(A_k,\psi_k).$ It is known such defined
$A$ is a unital simple \CA. Moreover,
$$
(K_0(A), (K_0(A))_+)=(G, G_+).
$$
Consider the unitaries
$$
u_j=e^{i\theta_0}p_{1,j}+e^{-i\theta_0/(n-1)}p_{2,j}+p_{3,j},
$$
where $\{p_{1,j}, p_{2, j}, p_{3,j}\}\subset M_{r(1,j)}$ are mutually orthogonal constant projections, $p_{1,j}$ has rank $d(j),$ $p_{2,j}$ has rank $(n-1)r(j)$ and
$p_{3,j}$ has rank $k(j)<n,$ $j=1,2,...,m(1).$
Define
\beq\label{2ml-2}
w=u_1\oplus u_2\oplus \cdots u_{m(1)}.
\eneq
Let $u=\psi_{1, \infty}(w),$ where $\psi_{1, \infty}$ is the \hm\, induced by the inductive limit system. Since each $u_j\in CU(C([0,1], M_{r(1,j)}),$ $u\in CU(A).$
We now verify that $u$ satisfies the assumption.
Let $\dt_1>0$ be as in \ref{Ex2leng} for $\ep/2$ (in place of $\ep$)
and $k_0=n-1.$ Let $\dt=\dt_1/2.$
Suppose that
there is a self-adjoint  element $h\in A_{s.a.}$ with $\|h\|\le 2\pi$ such that
\beq\label{2ml-3}
\|u-\exp(ih)\|<\dt.
\eneq
There is, for a sufficiently larger $k,$  a self-adjoint  element $h_1\in A_k$ for some $k\ge 1$ such
that
\beq\label{2ml-4}
\|h-\psi_{k, \infty}(h_1)\|<\ep/4\andeqn \|\psi_{1, k}(w)-\exp(ih_1)\|<2\dt=\dt_1.
\eneq
Consider a summand $A_{i,k}$ of $A_k.$
Note that $A_{i,k}=C([0,1], M_{r(i,k)}).$
We compute that
\beq\label{2ml-5}
\psi_{1,k}(w)=
e^{i\theta_0}P_{1,i,k}+e^{-i \theta_0/(n-1)}P_{2,i,k}+v_{0,i},
\eneq
where  $v_{0,i}\in P_{3,i,k}E_{i,k}P_{3,i,k}$ is a constant
unitary, $P_{1,i,k}, P_{2,i,k}, P_{3,i,k}$  are mutually orthogonal projections with
$$
P_{1,i,k}+P_{2,i,k}+P_{3,i,k}={\rm id}_{A_{i,k}},
$$
$P_{2,i,k}$ has rank $n-1$ times as much as $P_{1,i,k}$
and $P_{1,i,k}$ has rank at least $m_0$ times that of the rank of $P_{3,i,k}.$
Denote by $K_0$ the rank of $P_{3, i,k}.$ Then we have
$$
{\rm rank}P_{1,i,k}>2^{15} n^3(K_0+1)\pi^2.
$$
It follows from \ref{Ex2leng} that
\beq\label{2ml-6}
\|h_1\|\ge 2(1-1/(n-1)\pi=2\pi-2\pi/(n-1)\ge 2\pi-\ep/2
\eneq
Note that each $\psi$ is injective. Therefore
$$
\|\psi_{k, \infty}(h_1)\|=\|h_1\|\ge 2\pi-\ep/2.
$$
By (\ref{2ml-4}),
$$
\|h\|\ge 2\pi-\ep.
$$

\end{proof}

\begin{thm}\label{L2pi}
Let $(G_0, (G_0)_+)$ be a countable weakly unperforated Riesz group and let $G_1$ be any countable abelian group. There exists a unital
simple AH-algebra $A$ with tracial rank one such that
$$
(K_0(A), K_0(A)_+, K_1(A))=(G_0, (G_0)_+, G_1).
$$
Moreover, for any $\ep>0,$ there exists a unitary $u\in CU(A)$ and there exists $\dt>0$ satisfying the following:
If $h\in A_{s.a.}$ such that
$$
\|u-\exp(ih)\|<\dt,
$$
then
$$
\|h\|\ge 2\pi-\ep.
$$

\end{thm}

\begin{proof}
By \cite{Jasp}, there exists a unital simple AH-algebra $C$
with no dimension growth such that
$$
(G_0, (G_0)_+, G_1)=(K_0(C), K_0(C)_+, K_1(C)).
$$
Let $F=\rho_C(K_0(C)).$ Let
$B$ be a unital simple \CA\, which is an inductive limit of interval algebras such that
$$
(K_0(B), K_0(B)_+)=(F, F_+)
$$
which satisfies the conclusion of \ref{Ex2ml}.
Denote by $S_{[1_B]}(K_0(B))$ the state space of $K_0(B).$  Denote by  $r: T(B)\to S_u(K_0(B))$ the map
defined by $r(\tau)([p])=\tau(p)$ for all $\tau\in T(B)$ and for all
projections in $M_{\infty}(B).$
Let
$$
(G_0, (G_0)_+, [1_B], G_1, T(B), r)
$$
be a six-tuple of the Elliott invariant.
Let $A$ be a unital simple AH-algebra with no dimension growth
such that
$$
(K_0(A), K_0(A)_+, [1_A], K_1(A), T(A), r_A)
=((G_0, (G_0)_+, [1_B], G_1, T(B), r)
$$
(given by, say (\cite{Jasp})).
Fix $\ep>0.$
Let $u_1\in CU(B)$ (in place of $u$) which has the property as stated in \ref{Ex2ml} for $\ep/2$ given above. Let $\dt_1>0$ (in place of $\dt$) be the corresponding number. Let $\dt=\dt_1/2.$
Let $\Gamma: C(\T)_{s.a.}\to \Aff(T(B))$ be the map induced
by $u_1.$
Note that $\Aff(T(B))=\Aff(T(A)).$
By 8.4 of \cite{Lninv}, there is a unitary $u\in CU(A)$ such that
$$
\tau(f(u))=\Gamma(f)(\tau)\tforal f\in C(\T)_{s.a.}.
$$
Now let
$\phi: A\to B$ be a unital \hm\,
such that $\phi_{*0}={\rm id}_{G_0}$ and
$\psi_{\sharp}: T(B)\to T(A)$ is the identity map (when we identify
$T(A)$ and $T(B)$) and
$r_A(\psi_{\sharp}(t))([p])=r(t)(\psi_{*0}([p]))$ for all
$t\in T(B)$ and all projections $p\in M_{\infty}(A).$
Note that
$$
\tau(f(\psi(u)))=\tau(f(u_1))\tforal \tau\in T(B)
$$
for all $f\in C(\T)_{s.a.}.$ Moreover
$\psi(u)\in CU(B).$ It follows from  \ref{QT} (see \cite{Lnappeqv})   that
there exists a sequence of unitaries $\{w_n\}\subset B$ such that
\beq\label{2LL-1}
\lim_{n\to\infty}w_n^*\psi(u)w_n=u_1.
\eneq
Suppose that there is a self-adjoint  element $h\in A$ such that
$$
\|u-\exp(ih)\|<\dt.
$$
Then, for some large $n,$
$$
\|u_1-\exp(i w_n^*hw_n)\|<\dt_1.
$$
It follows from that
$$
\|w_n^*hw_n\|\ge 2\pi-\ep.
$$
Thus
$$
\|h\|\ge 2\pi-\ep.
$$

\end{proof}

\begin{cor}\label{Lastcor}
Let $(G_0, (G_0)_+)$ be a countable weakly unperforated Riesz group and let $G_1$ be any countable abelian group. There exists a unital
simple AH-algebra $A$ with tracial rank one such that
\beq
(K_0(A), K_0(A)_+, K_1(A))&=&(G_0, (G_0)_+, G_1)\andeqn\\
{\rm cel}_{CU}(A)&>&\pi.
\eneq
\end{cor}

\begin{proof}
Let $A$ be in the conclusion of \ref{L2pi}.
Let $\ep=\pi/16.$
Choose a unitary $u$ in $A$ and $\dt$ satisfy the conclusion of \ref{L2pi} for this $\ep.$
We may assume that $\dt<1/64.$
We will show that ${\rm cel}(u)>\pi.$
Otherwise, one obtains a self-adjoint  element
$h\in A$ with $\|h\|\le \pi$ such that
$$
\|u-\exp(ih)\|<\dt.
$$
This is not possible.

\end{proof}

\bibliographystyle{plain}

\begin{thebibliography}{10}


\bibitem{BPT} N.  Brown, F. Perera and A.  Toms, {\em The Cuntz semigroup, the Elliott conjecture, and dimension functions on \CA s},  J. Reine Angew. Math., {\bf 621} (2008), 191--211.

\bibitem{GLexp} G.  Gong and H. Lin, {\em The exponential rank of inductive limit \CA s},  Math. Scand., {\bf 71} (1992),  301--319.

\bibitem{GLN} G. Gong, H. Lin and Z. Niu, {\em
Classification of simple \CA s of generalized tracial rank one},
in preparation.

\bibitem{JS}X.  Jiang and H.  Su, {\em On a simple unital projectionless \CA}, Amer. J. Math., {\bf  121} (1999), 359--413.

\bibitem{Lnexp} H.  Lin, {\em Exponential rank of \CA s with real rank zero and the Brown-Pedersen conjectures}, J. Funct. Anal., {\bf  114} (1993),  1--11.
\bibitem{Lnproclond} H. Lin, {\em  The tracial topological rank of \CA s}, Proc. London Math. Soc., {\bf 83} (2001), 199--234.

\bibitem{Lnbk} H. Lin, {\em  An introduction to the classification of amenable \CA s},  World Scientific Publishing Co., Inc., River Edge, NJ, 2001. xii+320 pp. ISBN: 981-02-4680-3.

\bibitem{Lntr1} H.  Lin, {\em Simple nuclear \CA s of tracial topological rank one},  J. Funct. Anal., {\bf 251} (2007), 601--679.

\bibitem{Lnhuasdo} H.  Lin, {\em Inductive limits of subhomogeneous \CA s with Hausdorff spectrum},  J. Funct. Anal.£¬ {\bf 258} (2010),  1909--1932.

\bibitem{Lnexp2} H.  Lin, {\em Unitaries in a simple \CA\, of tracial rank one},  Internat. J. Math., {\bf  21} (2010),  1267--1281.



\bibitem{Lnhomt1} H.  Lin, {\em Homotopy of unitaries in simple \CA s with tracial rank one},  J. Funct. Anal., {\bf 258} (2010), 1822--1882.

\bibitem{Lninv} H.  Lin, {\em Asymptotic unitary equivalence and classification of simple amenable \CA s},  Invent. Math. {\bf 183} (2011), 385--450.

\bibitem{Lnappeqv} H.  Lin, {\em Approximate unitary equivalence in simple \CA s of tracial rank one},  Trans. Amer. Math. Soc. {\bf 364} (2012),  2021--2086.

\bibitem{Lnappen} H. Lin, {\em Localizing the Elliott Conjecture at Strongly Self-absorbing C*-algebras, II} --An Appendix,
    J. Reine Angew. Math., to appear ( arXiv:0709.1654).



\bibitem{LN-1} H.  Lin and Z.  Niu, {\em Lifting KK-elements, asymptotic unitary equivalence and classification of simple \CA s},  Adv. Math.,  {\bf 219} (2008),  1729--1769.

\bibitem{LN-2} H.  Lin and Z.  Niu, {\em The range of a class of classifiable separable simple amenable \CA s},  J. Funct. Anal.,
    {\bf 260},  (2011),  1--29.

\bibitem{LN-3} H. Lin and Z. Niu, {\em  Homomorphisms into a simple Z-stable C*-algebras}, preprint (arXiv:1003.1760).

\bibitem{LS} H. Lin and W. Sun, {\em Tensor products of classifiable \CA s}, preprint (arXiv:1203.3737).


\bibitem{Ri} J. R.  Ringrose, {\em Exponential length and exponential rank in \CA s},  Proc. Roy. Soc. Edinburgh Sect. A {\bf 121} (1992), 55--71.
\bibitem{RW} M.  R\o rdam and W. Winter, {\em The Jiang-Su algebra revisited}, J. Reine Angew. Math., {\bf 642} (2010), 129--155.

\bibitem{RP} N. C.  Phillips and J. R.  Ringrose, {\em Exponential rank in operator algebras, Current topics in operator algebras},  (Nara, 1990), 395--413, World Sci. Publ., River Edge, NJ, 1991.

\bibitem{Ph-fu} N. C.  Phillips, {\em Simple \CA s with the property weak (FU)},  Math. Scand. {\bf 69} (1991),  127--151.

\bibitem{Ph-2} N. C. Phillips, {\em Approximation by unitaries with finite spectrum in purely infinite \CA s}, J. Funct. Anal., {\bf 120} (1994), 98--106.

\bibitem{Ph-3} N. C.   Phillips, {\em Reduction of exponential rank in direct limits of \CA s},  Canad. J. Math., {\bf 46} (1994), 818--853.


\bibitem{Ph-4} N. C.  Phillips, {\em How many exponentials?} Amer. J. Math., {\bf 116} (1994), 1513--1543.





\bibitem{Ph-1} N. C.  Phillips, {\em Exponential length and traces},  Proc. Roy. Soc. Edinburgh Sect. A  {\bf 125} (1995), 13--29.

\bibitem{Thoms-2}K. Thomsen, {\em Diagonalization in inductive limits of circle algebras}, Journal of Operator Th., {\bf 27} (1992) 325--340.

\bibitem{Thoms-1} K.  Thomsen, {\em On the reduced $C^*$-exponential length},  Operator algebras and quantum field theory (Rome, 1996), 59--64, Int. Press, Cambridge, MA, 1997.




\bibitem{Jasp} J.  Villadsen, {\em The range of the Elliott invariant of the simple AH-algebras with slow dimension growth}, K-Theory, {\bf 15} (1998),  1--12.
\bibitem{W} W. Winter {\em Localizing the Elliott conjecture at strongly self-absorbing C*-algebras}, J. Reine Angew. Math.,to appear, (arXiv:0708.0283).

\bibitem{Zh-1} S. Zhang, {\em On the exponential rank and exponential length of \CA s}, J. Operator Theory, {\bf 28} (1992),  337--355.

\bibitem{Zh-2} S. Zhang, {\em Exponential rank and exponential length of operators on Hilbert \CA s},  Ann. of Math. {\bf 137} (1993),  121--144.
\end{thebibliography}

\vspace{0.2in}

hlin@uoregon.edu

\end{document}